\documentclass[12pt]{amsart}
\usepackage{amsmath}
\usepackage[utf8]{inputenc}

\usepackage{svg} 
\usepackage[off]{svg-extract}
\svgsetup{clean=true}

\usepackage{psfrag}

\usepackage{enumerate}
\usepackage[shortlabels]{enumitem}
\usepackage{amsfonts}
\usepackage{kantlipsum}
\usepackage{float}

\usepackage{pdftexcmds} 

\usepackage{booktabs}
\usepackage{amsmath,amsthm,amsfonts,latexsym,amscd,amssymb,enumerate,color,hyperref}
\usepackage{mathscinet}
\input{xypic}

\usepackage{fullpage}

\usepackage{rotating}
\usepackage{multirow}
\usepackage{graphicx}
\usepackage{hyperref}
\usepackage{epstopdf}
\usepackage{inputenc}
\usepackage{tikz-cd}

\usepackage{soul} 
\usepackage[normalem]{ulem}

\usepackage[colorinlistoftodos, 
textwidth=2.5cm]{todonotes}

\newcounter{lgcomment}

\newcounter{lacomment}

\newcounter{fgcomment}

\usepackage{booktabs}
\usepackage{enumitem}
\usepackage{multicol}

\usepackage{hyperref}
\hypersetup{
	colorlinks=true,
	linkcolor=blue,
	citecolor=blue,
	urlcolor=blue,
	pdfauthor={},
	pdftitle={}
}

\usepackage{lineno}

\def\r{\mathbb{R}}
\def\s{\mathbb{S}}
\def\t{\mathbb{T}}

\def\cA{\mathcal{A}}

\def\cP{\mathcal{P}}

\def\pr2{\r P^2}
\def\Susp{\mathrm{Susp}}
\def\Kl{\mathrm{Kl}}
\def\pt{\mathrm{pt}}
\DeclareMathOperator{\interior}{int}

\newcommand{\TP}{T\hspace{-2.9pt}P}

\newcommand{\SCS}{\widehat{\#}} 

\numberwithin{equation}{section}

\usepackage{amsthm}

\newtheorem{theorem}{Theorem}[section]
\newtheorem{lemma}[theorem]{Lemma}
\newtheorem{proposition}[theorem]{Proposition}

\newtheorem{thm}{Theorem}

\newtheorem{cor}[thm]{Corollary}

\theoremstyle{remark}
\newtheorem{remark}[theorem]{Remark}

\theoremstyle{definition}
\newtheorem{definition}[theorem]{Definition} 

\theoremstyle{remark} 
\newtheorem*{ack}{Acknowledgements}

\usepackage{graphicx}
\graphicspath{{img/}}

\begin{document}

\title{Decompositions of three-dimensional Alexandrov spaces}

\author[Franco Reyna]{Luis Atzin Franco Reyna$^{(1,2)}$}
\address[Franco Reyna]{Facultad de Ciencias, UNAM, Mexico and Department of Mathematical Sciences, Durham University, United Kingdom and Department of Mathematics, University of Notre Dame, USA.}
\email{lfrancor@nd.edu}

\thanks{$^{(1)}$Supported in part by a scholarship ``Iniciaci\'on a la Investigaci\'on'' from UNAM, Mexico and hosted by Durham University.}
\thanks{$^{(2)}$Supported in part by NSF grant DMS-2042303 ``Singular Riemannian Foliations and Applications to Curvature and Invariant Theory''.}

\author[Galaz-Garc\'ia]{Fernando Galaz-Garc\'ia$^{(3,4)}$}
\address[Galaz-Garc\'ia]{Department of Mathematical Sciences, Durham University, United Kingdom.}
\email{f.galaz-garcia@durham.ac.uk}

\thanks{$^{(3)}$Supported in part by the Deutsche Forschungsgemeinschaft (DFG) grant GA 2050 2-1 within the Special Priority Programme 2026 ``Geometry at Infinity'' and grant no. 281869850 under the Research Training Group 2229 ``Asymptotic Invariants and Limits of Groups and Spaces''.}

\thanks{$^{(4)}$Supported in part by research grants MTM2017-‐85934-‐C3-‐2-‐P from the  Ministerio de Econom\'{\i}a y Competitividad de Espa\~{n}a (MINECO) and PID2021-124195NB-C32 from the Ministerio de Ciencia e Innovación (MICINN), and by ICMAT Severo Ochoa project CEX2019-000904-S (MINECO), Spain.}

\author[G\'omez-Larra\~{n}aga]{Jos\'e Carlos G\'omez-Larra\~naga}
\address[G\'omez-Larra\~naga]{CIMAT-M\'erida, Yucatan, Mexico.}
\email{jcarlos@cimat.mx}


\author[Guijarro]{Luis Guijarro$^{(4)}$}
\address[Guijarro]{Department of Mathematics, Universidad Aut\'onoma de Madrid and ICMAT CSIC-UAM-UC3M, Spain}
\email{luis.guijarro@uam.es}

\author[Heil]{Wolfgang Heil}
\address[Heil]{Department of Mathematics, Florida State University, USA.}
\email{heil@math.fsu.edu}

\setlist{wide, labelwidth=0pt, labelindent=0pt}

\begin{abstract}
We extend basic results in $3$-manifold topology to general three-dimensional Alexandrov spaces (or Alexandrov $3$-spaces for short), providing a unified framework for manifold and non-manifold spaces. We generalize the connected sum to non-manifold $3$-spaces and prove a prime decomposition theorem, exhibit an infinite family of closed, prime non-manifold $3$-spaces which are not irreducible, and establish a conjecture of Mitsuishi and Yamaguchi on the structure of closed, simply-connected Alexandrov $3$-spaces with non-negative curvature. Additionally, we define a notion of generalized Dehn surgery for Alexandrov $3$-spaces and show that any closed Alexandrov $3$-space may be obtained by performing generalized Dehn surgery on a link in $S^3$ or the non-trivial $S^2$-bundle over $S^1$. As an application of this result, we show that every closed Alexandrov $3$-space is homeomorphic to the boundary of a $4$-dimensional Alexandrov space.
\end{abstract}
\date{\today}


\maketitle





\section{Introduction}
\label{s:intro}
Alexandrov spaces are complete, locally compact length spaces with finite (integer) Hausdorff dimension and curvature bounded below in the triangle comparison sense. They are metric generalizations of complete Riemannian manifolds with sectional curvature uniformly bounded below and were first studied by Burago, Gromov and Perelman in \cite{BGP}. Alexandrov spaces play an important role in global Riemannian geometry. There, they arise as orbit spaces of isometric compact Lie group actions on Riemannian manifolds with sectional curvature bounded below or as Gromov--Hausdorff limits of sequences of Riemannian $n$-manifolds with a uniform lower sectional curvature bound. In the latter guise, they appear, for example, in Perelman's proof of Thurston's geometrization conjecture \cite{Perelman.Poincare.1,Perelman.Poincare.2}. Infinite-dimensional Alexandrov spaces also arise in applications of metric geometry to data analysis, where geometric considerations of data sets come into play \cite{Bubenik.Hartsock,Che.et.al,Turner.et.al}. 
Furthermore, Alexandrov spaces have a rich geometric and topological structure, making them an interesting subject of study in their own right.

Alexandrov spaces are not necessarily homeomorphic to manifolds (they include, for example, Riemannian orbifolds with sectional curvature uniformly bounded below) and their topology is far from being understood. To address this problem, it is natural to consider first low-dimensional spaces. By the work of Perelman, one- and two-dimensional Alexandrov spaces are homeomorphic to topological manifolds (see \cite{bbi}). The present article focuses on three-dimensional Alexandrov spaces, or \emph{Alexandrov $3$-spaces} for short.

A closed (i.e., compact and without boundary) Alexandrov $3$-space is homeomorphic to either a topological $3$-manifold or a non-orientable topological $3$-manifold with boundary an even number of copies of the real projective plane $P^2$ which are ``capped off'' by gluing cones over $P^2$. Conversely, any such space is homeomorphic to some Alexandrov space (see \cite{galaz-guijarro2015}). Locally, every point in an Alexandrov $3$-space has a neighborhood homeomorphic to either a $3$-ball or a cone over $P^2$. Thus, the geometry and topology of Alexandrov $3$-spaces include those of $3$-manifolds as a particular case. The non-manifold case has only been explored recently (see, for example, \cite{barcenas.nunez.zimbron.2021,barcenas.sedano.2023,Deng.et.al,galaz-guijarro2020,galaz.garcia.nunez.zimbron.2020.local.S1.actions,nunez.zimbron.2018.S1.actions,galaz-garcia.searle.2011,galaz-garcia.tuschmann.2019,mitsuishi-yamaguchi2015} and the surveys \cite{galaz.garcia.2016.survey,galaz.guijarro.nunez.zimbron,galaz.nunez.zimbron}) and finds applications within and without  metric geometry (see \cite{foscolo,grove_petersen}). 
Interestingly, non-manifold Alexandrov spaces are homeomorphic to spaces that have appeared previously in the literature under the name of ``singular $3$-manifolds'', introduced by Quinn in \cite{Quinn81} (see also \cite{HAM,HAS}). Nevertheless, a theory for singular $3$-manifolds analogous to that for $3$-manifolds still needs to be fully developed (see \cite[Open Problem 6]{MaRo}). Motivated by the preceding considerations, we extend basic results in $3$-manifold topology to general Alexandrov $3$-spaces, providing a unified framework for manifold and non-manifold spaces, and derive some geometric conclusions.

First, we generalize the notion of connected sum to the non-manifold case (see Definition~\ref{def:connected.sum}). For closed non-manifold Alexandrov $3$-spaces $P$ and $Q$, we may remove open neighborhoods $U_p$ and $V_q$ of points $p\in P$ and $q\in Q$, respectively, so that both $U_p$ and $V_q$ are homeomorphic to a ball or a cone over $P^2$. We then identify the boundary components of $P\setminus U_p$ and $Q\setminus V_q$ to obtain a space $P\#^{p,q}Q$. In contrast to the manifold case, the resulting space may depend on the choice of points used to perform the connected sum.


\begin{thm}
\label{thm:connected.sum.depends.on.points}    
There exist closed non-manifold Alexandrov $3$-spaces $P$, $Q$ and points $p_1,p_2\in P$, $q_1,q_2\in Q$ such that $P\#^{p_1,q_1}Q$ and  $P\#^{p_2,q_2}Q$ are not homeomorphic. 
\end{thm}

We next generalize the notion of prime manifold (see Definition~\ref{def:prime.space}) and prove that every closed  Alexandrov $3$-space admits a connected sum decomposition into prime spaces. This result generalizes the classical prime decomposition theorem for $3$-manifolds of Kneser (see, for example, \cite{hempel}).  


\begin{thm}
\label{thm:existence.prime.decomposition}
Every closed Alexandrov $3$-space has a prime decomposition.
\end{thm}

Irreducibility for general Alexandrov spaces was defined in \cite{galaz.guijarro.nunez.zimbron}, taking into account the presence of topological singularities in the non-manifold case to ensure consistency with the definition of irreducibility for $3$-manifolds (see Definition~\ref{def:irreducibility}).
As in the manifold case, every irreducible Alexandrov space is prime (see Proposition~\ref{prop:irreducible.then.prime}). 
In Theorem~\ref{thm:connected.sum.depends.on.points}, we see that permuting singular summands in a prime decomposition  with respect to the non-manifold connected sum might result in 
non-homeomorphic spaces. This highlights
the importance of understanding not only the prime spaces resulting from a decomposition of the space but also the specific manner in which they are glued  
together. This stands in stark contrast to the manifold case, where Milnor established the uniqueness of the prime decomposition for closed $3$-manifolds, up to permutation of prime factors (see \cite{milnor}). For general closed Alexandrov $3$-spaces, we may first decompose the space using the manifold connected sum and then further reduce each resulting piece using the singular connected sum. A decomposition of this form will be called a \emph{normal} prime decomposition (see Definition~\ref{def:normal.prime.decomposition}).
We show that such decompositions are unique.
This may be seen as a stronger form of uniqueness compared to that in the manifold prime decomposition, as not only the 
prime factors are unique, but also the way they are glued together.

\begin{thm}
\label{thm:normal.decomposition.uniqueness}
Every closed Alexandrov $3$-space admits a unique normal prime decomposition.
\end{thm}

The presence of topological singularities leads to new and interesting topological phenomena. We show that, in contrast to the manifold case,  where 
every prime $3$-manifold is irreducible except for the non-orientable $S^2$-bundle over $S^1$ and $S^2 \times S^1$, there exist a infinitely many closed, prime Alexandrov $3$-spaces that are not irreducible.


\begin{thm}
\label{thm:infinite.family.prime.not.irreducible}
There exists an infinite family of mutually non-homeomorphic non-manifold Alexandrov $3$-spaces which are prime and are not irreducible.
\end{thm}

Our next main result establishes a conjecture of Mitsuishi and Yamaguchi asserting that every closed, simply-connected Alexandrov $3$-space with non-negative curvature can be obtained by gluing together two spaces coming from a list of only five different possible non-negatively curved Alexandrov spaces with boundary (see \cite[Conjecture 1.10]{mitsuishi-yamaguchi2015}). The corresponding conjecture in the positively curved case, asserting that a closed, simply-connected positively curved Alexandrov $3$-space is homeomorphic to the $3$-sphere or to $\Susp(P^2)$, the suspension of $P^2$ (see \cite[Conjecture 1.11]{mitsuishi-yamaguchi2015}) was settled in \cite{galaz-guijarro2015}.


\begin{thm}
\label{thm:mitsuishi.yamaguchi}
A closed, simply-connected Alexandrov $3$-space with non-negative curvature is homeomorphic to an isometric gluing $A\cup_{\partial} A'$ for $A$ and $A'$ chosen in the following list of non-negatively curved Alexandrov spaces:
\begin{equation*}
    D^3,\ K(P^2),\ B(\mathrm{pt}),\ B(S_2),\ B(S_4).
\end{equation*}
\end{thm}

The isometric gluings appearing in Theorem~\ref{thm:mitsuishi.yamaguchi} are homeomorphic to one of the following four spaces: $S^3$, $\Susp(P^2)$, $\Susp(P^2)\#\Susp(P^2)$, or the capped octopod, a certain quotient of the flat $3$-torus. We precisely define the spaces listed in Theorem~\ref{thm:mitsuishi.yamaguchi}, along with other spaces, including the octopod, in Section~\ref{s:preliminaries}. The proof of Theorem~\ref{thm:mitsuishi.yamaguchi} is based on the classification of closed Alexandrov $3$-spaces with non-negative curvature in \cite{galaz-guijarro2015}. 

 By the work of Lickorish \cite{Lickorish1962,Lickorish1963} and Wallace \cite{Wallace1960}, every closed $3$-manifold can be obtained by Dehn surgery on a knot in the $3$-sphere or on the non-trivial $S^2$-bundle over $S^1$. One may use these results to show that every closed $3$-manifold is the boundary of a $4$-dimensional manifold (see \cite{Lickorish1962,Lickorish1963} and cf.\ \cite{Thom}). We define a notion of generalized Dehn surgery for Alexandrov $3$-spaces and use it to obtain analogues of Lickorish's results in this case. 


\begin{thm}
\label{thm:lickorish.dehn.surgery}
Any closed Alexandrov $3$-space may be obtained by generalized Dehn surgery on a link either in the $3$-sphere or in the non-trivial $S^2$-bundle over $S^1$.
\end{thm}


\begin{cor}
\label{cor:lickorish.boundary.4d}
Every closed Alexandrov $3$-space is homeomorphic to the boundary of a $4$-dimensional Alexandrov space.
\end{cor}

 Our article is organized as follows. In Section~
 \ref{s:preliminaries}, we recall basic facts on Alexandrov $3$-spaces. In Section~\ref{s:prime.decomposition}, we define the connected sum, prime decomposition, and prove Theorem~\ref{thm:existence.prime.decomposition}. Section~\ref{s:infinitely.many.prime.non.irreducible.spaces} contains the proof of  Theorem~\ref{thm:infinite.family.prime.not.irreducible}. 
We prove Theorem~\ref{thm:mitsuishi.yamaguchi} in Section~\ref{s:mitsuishi.yamaguchi}.
 Finally, in Section~\ref{s:dehn.surgery}, we define generalized Dehn surgery and prove Theorem~\ref{thm:lickorish.dehn.surgery} and Corollary~\ref{cor:lickorish.boundary.4d}.

\begin{ack} Wolgang Heil and Jos\'e Carlos G\'omez-Larrañaga thank the Institute for Algebra and Geometry of the Karlsruhe Institue of Technology (KIT) for its hospitality while part of the work in the present article was carried out. Luis Atzin Franco Reyna, Jos\'e Carlos G\'omez-Larrañaga, and Luis Guijarro thank the Department of Mathematical Sciences of Durham University for its hospitality during the course of the project. 
\end{ack}


\section{Preliminaries} 
\label{s:preliminaries}

In this section, we collect basic facts on three-dimensional Alexandrov spaces (or \emph{Alexandrov $3$-spaces} for short) and list some special spaces that we will use in the rest of the article. For basic material on metric geometry and $3$-manifolds, we refer the reader to \cite{bbi} and \cite{hempel}, respectively. For a broader discussion of Alexandrov $3$-spaces, we refer the reader to \cite{galaz.guijarro.nunez.zimbron,galaz.nunez.zimbron}.


\subsection{Three-dimensional Alexandrov spaces}
\label{ss:3d.alexandrov.spaces}

A locally compact geodesic space $(X,d)$ of finite (Hausdorff) dimension is an \emph{Alexandrov space} (with curvature bounded below) if it has curvature bounded below in the triangle comparison sense. We refer the reader to \cite{bbi,BGP} for the main definitions and theorems about these spaces. For further results, see \cite{Petrunin2007} and \cite{AKP}. Local compactness and the lower curvature bound imply that the Hausdorff dimension must be a non-negative integer. One- and two-dimensional Alexandrov spaces are, respectively, homeomorphic to curves and surfaces. Starting in dimension three, however, topological singularities may appear. For example, one may consider the iterated spherical suspension of a round real projective plane, which is an Alexandrov space with curvature bounded below by $1$ where the vertices of the suspension are non-manifold points.

We will denote the real projective space by $P^2$. The symbol $\approx$ will denote homeomorphism between topological spaces. The suspension of a topological space $T$ will be denoted by $\Susp(T)$. 
We will denote the set of all closed (i.e., compact and without boundary) Alexandrov $3$-spaces  by $\cA$. We will usually denote an arbitrary space in $\cA$ by $P$ or $Q$. Given a space $P\in \cA$, we can associate to each point $p\in P$ a $2$-dimensional Alexandrov space with curvature bounded below by $1$, called the \emph{space of directions of $P$ at $p$} and denoted by $\Sigma_x P$. When there is no risk of confusion, we will write $\Sigma_p$. The space of directions of a point in $P$ is homeomorphic to either the $2$-sphere $S^2$ or the real projective plane $P^2$. By Perelman's conical neighborhood, a sufficiently small open neighborhood of a point $p\in P$ is homeomorphic to $K_p$, the Euclidean cone over the space of directions $\Sigma_p$. If every point in $P$ has space of directions homeomorphic to $S^2$, then $P$ is a $3$-manifold. We will say that $P$ is \emph{topologically singular} if it contains a point whose space of directions is $P^2$. We will call such a point a \emph{topologically singular point}. Points whose space of directions is homeomorphic to $S^2$ will be called \emph{topologically regular}. We will denote topologically singular spaces in $\cA$ by $X$ or $Y$. Manifolds will be denoted by $M$ and $N$. 

   Let $X\in \cA$ be a topologically singular space. We will denote by $M_X$ the $3$-manifold with boundary obtained by removing disjoint open regular neighborhoods of the topologically singular points of $X$. We may write $X$ as the union of $M_X$ and finitely many copies of a closed cone over $P^2$. In other words, we obtain $X$ from $M_X$, a compact non-orientable $3$-manifold with finitely many $P^2$ boundary components, by capping off the boundary components with cones over $P^2$. It is not difficult to see that $M_X$ must have an even number of boundary components (see, for example, \cite[Proof of Theorem 9.5]{hempel}). Thus, $X$ must have an even number of topologically singular points. We may also exhibit $X$ as the quotient of a closed, orientable $3$-manifold by a piecewise-linear (PL) orientation-reversing involution.
   

\begin{proposition}[cf.\ \protect{\cite[Lemma 1.7]{galaz-guijarro2015}}]
\label{prop:double_branched_cover}
Let $X$ be a closed Alexandrov $3$-space. If $X$ is not a topological manifold, then there is a closed orientable $3$-manifold $\tilde{X}$ and an orientation-reversing involution $\iota\colon \tilde{X}\to \tilde{X}$ with only isolated fixed points such that $X$ is homeomorphic to the quotient $\tilde{X}/\iota$. Moreover, the involution $\iota$ is equivalent to a PL involution on $\tilde{X}$.
\end{proposition}

We will call the $3$-manifold $\tilde{X}$ in Proposition~\ref{prop:double_branched_cover} the \emph{orientable double branched cover} of $X$. The corresponding branched covering map, given by the orbit projection map of the involution $\iota$, will be denoted by $\mathsf{p}\colon \widetilde{X}\to X$.
Observe that the fixed points of the involution on $\tilde{X}$ correspond to the topologically singular points of $X$.


\subsection{Three-dimensional blocks}
\label{S:3-dimensional.blocks}

Let us now define some special spaces which will often appear in subsequent sections. We divide them into two families: \emph{manifold blocks}, which are certain compact $3$-manifolds  with non-empty boundary, and \emph{singular blocks}, which are non-negatively curved Alexandrov spaces with non-empty boundary and topological singularities. The first family appeared in \cite{heil-larranaga} while the second one appeared in \cite{galaz-guijarro2015,mitsuishi-yamaguchi2015}.


\subsubsection{Manifold blocks}
\label{ss:special.manifolds}

Following \cite{heil-larranaga}, we will describe some compact $3$-manifolds whose boundary components may be tori, Klein bottles, or projective planes. In these examples, we will have a compact orientable $3$-manifold $M$ and an orientation-reversing involution $\tau\colon M\to M$ with $m>0$ fixed points. Then, we will choose invariant $3$-ball neighborhoods $C_1,\ldots,C_m$ of the fixed points and let $M_\ast=\overline{M\setminus\left(C_1\cup\cdots\cup C_m\right)}/\tau$ be the orbit manifold.
\\

\begin{enumerate}[label=(\roman*)]

    
    \item The \textit{geminus}. Let $M = D^2\times S^1$ and consider the orientation-reversing involution $\tau\colon
    D^2\times S^1\to D^2\times S^1$ given by $\tau(x, z) = (-x,\overline{z})$ with two isolated fixed points.
    After removing the balls as indicated above, we obtain the \textit{geminus} as the quotient $M_* = (P^2\times I)\#_b(P^2\times I)$, the boundary connected sum
    of two copies of $P^2\times I$. We will denote this space by $G$. The boundary
    of the geminus consists of two projective planes and a Klein bottle. \\
    
    \item The \textit{dipus}. Let $M=(\Kl\tilde{\times} [-1,1])_o$ be the orientable twisted interval-bundle over the Klein bottle $\Kl$. The manifold $M$ can be obtained as the quotient of $T^2\times [-1,1]$ by the involution $\sigma\colon T^2\times[-1,1]\to T^2\times[-1,1]$ given by 
    $\sigma(z_1,z_2,t) = (-z_1,\overline{z}_2,-t)$. Note that the boundary of $M$ is a torus and $M$ can also be realized as the mapping cylinder of the double cover $p\colon T^2\to \Kl$. Consider now the involution on $M$ given by $\tau([z_1,z_2,t])=[-\overline{z}_1,-z_2,t]$, which has two fixed points. The orbit manifold $M_\ast$ is the \textit{dipus} and we will denote it by $D$; its boundary consists of two projective planes and one incompressible Klein bottle.\\
    
    \item The \textit{bipod}. Let $W_1$ and $W_2$ be two copies of $(\Kl\tilde{\times}[-1,1])_o$, the orientable twisted interval-bundle over the Klein bottle $\Kl$, 
    and let $M=W_1\cup_{\varphi}W_2$, the twisted double of $(Kl\tilde{\times}I)_o$, where
    \[\varphi\colon \partial W_1\to\partial W_2\]
    is given by $\varphi([z_1,z_2,1])=[z_2,z_1,1]$. The manifold $M$ is usually known as the \emph{Hantzche--Wendt manifold}. The involution $\tau\colon M \to M$ is given by $\tau([z_1,z_2,t])=[-\overline{z}_1,-z_2,t]$ if $(z_1,z_2,t)\in W_1$ and $\tau([z_1,z_2,t])=[-z_1,-\overline{z}_2,t]$ if $(z_1,z_2,t)\in W_2$. This involution has two fixed points. The orbit manifold $M_\ast$ is the \textit{bipod} and we will denote it by $B$; its boundary consists of two projective planes.\\


    \item The \textit{quadripus}. Let $M=T^2\times [-1,1]$ and consider the involution $\tau(z_1,z_2,t)=(\overline{z}_1,\overline{z}_2,-t)$, which has four isolated fixed points. The orbit manifold $M_\ast$ is the \textit{quadripus}. We will denote it by $Q$; its boundary consists of four projective planes and one incompressible torus.\\
    
    \item The \textit{tetrapod}. Let $M=T^2\times[-1,1]/\{(z_1,z_2,1)\sim(\overline{z}_1,\overline{z}_2,-1)\}$ and consider the involution $\tau([z_1,z_2,t])=[-\overline{z_1},\overline{z_2},-t]$, which has four isolated fixed points. The orbit manifold $M_\ast$ is the \textit{tetrapod} and we will denote it by $\TP$; its boundary consists of four projective planes.\\
    

    \item The \textit{octopod}. Let $M=T^3$ and consider the involution $\tau(z_1,z_2,z_3)=(\overline{z}_1,\overline{z}_2,\overline{z}_3)$ which has eight isolated fixed points. The orbit manifold $M_\ast$ is the \textit{octopod}. We will denote it by $O$; its boundary consists of eight projective planes.
\end{enumerate}


\subsubsection{Singular blocks} 
\label{spaces} 
We define the \emph{singular blocks} as follows (cf.\ \cite{galaz-guijarro2015,mitsuishi-yamaguchi2015}).
\\

\begin{enumerate}[label=(\roman*)]
    \item We let $K(P^2)$ be the Euclidean cone over a real projective plane $P^2$.\\
    \item We let $B(S_2)=S^2\times[-1,1]/(\sigma, -\mathrm{id})$, where $S^2$ is a sphere of non-negative curvature with an isometric involution $\sigma\colon S^2\to S^2$ topologically conjugate to the involution on the sphere given by the suspension of the antipodal map on the circle. The space $B(S_2)$ is homeomorphic to $\mathrm{Susp}(P^2)\setminus\mathrm{int}(D^3)$, where $D^3\subset\mathrm{Susp}(P^2)$  is a closed $3$-ball consisting of topologically regular points (see  \cite{mitsuishi-yamaguchi2015}).\\
    \item We let $B(S_4)=T^2\times[-1,1]/(\sigma, -\mathrm{id})$, where $T^2$ is a flat torus and the involution $\sigma\colon T^2\to T^2$ maps $(z_1,z_2)$ to $(\overline{z_1},\overline{z_2})$. Observe that $T^2/\sigma$ is homeomorphic to $S^2$. The space $B(S_4)$  has four topologically singular points, corresponding to the four fixed points of the involution $(\sigma,-id)$. One may see this by noting that at each fixed point, the differential of the involution acts as the antipodal map on the corresponding unit tangent sphere. The  space $B(S_4)$ has oriented double branched cover $T^2\times[-1,1]$ and   boundary $T^2$.\\
    \item We let $B(\mathrm{pt})=D^2\times S^1/\alpha$, where $\alpha$ is an isometric involution defined on $D^2\times S^1$ by
    \[
    \alpha((x,y),z)=((-x,-y),\overline{z}).
    \]
    The space $B(\mathrm{pt})$ has two singular points corresponding to the image in the quotient of the fixed points $((0,0),1), ((0,0),-1)\in D^2\times S^1$ (cf.\  \cite[Example 1.2]{mitsuishi-yamaguchi2015}).
    The space $B(\mathrm{pt})$ is homeomorphic to the boundary-connected sum of two copies of $K(P^2)$, and thus its boundary is a Klein bottle (see, for example, the remarks before Lemma 2.61 in \cite{mitsuishi-yamaguchi2015}). 
\end{enumerate}


\begin{remark}
We may cap off the $P^2$ boundary components of manifold blocks to obtain some singular blocks. For instance, the capped-off geminus is homeomorphic to $B(\mathrm{pt})$. We will use such relations later on. 
\end{remark}


\section{Prime decomposition}
\label{s:prime.decomposition}
In this section, we will prove Theorems \ref{thm:connected.sum.depends.on.points}, \ref{thm:existence.prime.decomposition}, and \ref{thm:normal.decomposition.uniqueness} on connected sums and prime decompositions of closed Alexandrov $3$-spaces. We first define the connected sum of such spaces and prove Theorem~\ref{thm:connected.sum.depends.on.points}. We then introduce some basic notions, establish some preliminary results, and prove Theorems~\ref{thm:existence.prime.decomposition} and ~\ref{thm:normal.decomposition.uniqueness}.
We note that there exist extensions of the Kneser--Milnor prime decomposition theorem   to $3$-manifolds with boundary (see \cite{gross.1969,gross.1979,heil,Przytycki.1979,Swarup.1970}) and to certain classes of $3$-orbifolds (see \cite{petronio.2007}).


\subsection{Connected sum and proof of Theorem~\ref{thm:connected.sum.depends.on.points}} Let us start by extending some fundamental notions for $3$-manifolds to general Alexandrov $3$-spaces.

\begin{definition}[Connected sum]
\label{def:connected.sum}
Let $P,Q\in\cA$, fix $p\in P$, $q\in Q$, and let $U_p\subset P$ and $V_q\subset Q$ be, respectively, open neighborhoods of $p$ and $q$ homeomorphic to $K_p$ and $K_q$, the tangent cones at $p$ and $q$, respectively. If $K_p\approx K_q$, then there is a homeomorphism
\begin{equation*}
\varphi\colon \partial(P\setminus U_p)\approx \Sigma_p \to \partial(Q\setminus V_q)\approx \Sigma_q
\end{equation*}
and we define the \textit{connected sum of $P$ and $Q$} as
\begin{equation}
P\#^{p,q}Q =\left((P\setminus U_p)\sqcup(Q\setminus V_q)\right)/_{w\sim \varphi(w)}.
\end{equation}
\end{definition}

To avoid cumbersome notation, when there is no risk of confusion regarding the points we have used to construct a connected sum of two spaces $P,Q\in \mathcal{A}$, we will write $P\# Q$ for the usual connected sum along balls and $P\hat{\#}Q$ for the connected sum along cones over $P^2$. Note, however, that the space $P\#^{p,q}Q$ may depend on the choice of points $p\in P$ and $q\in Q$ along which we take the connected sum. To see this, let us recall some notions from \cite{Heil-Negami1986,negami1981}, which we will use in the proof of Theorem~\ref{thm:connected.sum.depends.on.points}. 


Let $M$ be a $3$-manifold and $F_0$, $F_1$ closed $2$-manifolds embedded in $\interior(M)$, the interior of $M$. The surfaces $F_0$ and $F_1$ are \textit{parallel} if there is an embedding $h\colon F\times I\to M$ such that $h(F\times\{0\})=F_0)$ and $h(F\times\{1\})=F_1)$. A \textit{complete system of projective planes} in $M$ is a sytem $\mathcal{P}=\{P^2_1,\dots,P^2_q\}$ of mutually disjoint two-sided projective planes in $\interior(M)$ satisfying the following conditions:
\begin{enumerate}
    \item Every $P^2_i$ is not parallel to each other.
    \item If $P^2_{q+1}$ is a two-sided projective plane in $M$ disjoint from $P^2_1\cup\cdots\cup P^2_q$, then $P^2_{q+1}$ is parallel to some $P^2_i$ ($i=1,\cdots, q$).
\end{enumerate}


\begin{definition}[Colored $P^2$-graph of an irreducible compact $3$-manifold]
    Let $M$ be an irreducible compact $3$-manifold whose boundary is either empty or consists only of projective planes. Construct a colored graph $G(M)$ (embedded in $M$) as follows:
    \begin{enumerate}
        \item Choose a vertex $v_i$ in each component $C_i$ of $M$ cut open along a complete system $\mathcal{P}$ of two-sided projective planes in $\interior(M)$.
        \item Color $v_i$ white if $C_i=P^2\times I$ and one component of $\partial C_i$ is a component of $\partial M$. Otherwise, color $v_i$ black. 
 \item Join $v_i$ and $v_j$ by an edge if $C_i$ and $C_j$ contains a component of $\mathcal{P}$. 
    \end{enumerate}
    \end{definition}

 Note that the degree of a white vertex is $1$ and the degree of a black vertex is even. Figures~\ref{fig:colored.graph.closed} and \ref{fig:colored.graph.boundary} show examples of colored $P^2$-graphs.


\begin{figure}[htpb]
    \centering
	\def\svgwidth{.8\textwidth}
\begingroup%
  \makeatletter%
  \providecommand\color[2][]{%
    \errmessage{(Inkscape) Color is used for the text in Inkscape, but the package 'color.sty' is not loaded}%
    \renewcommand\color[2][]{}%
  }%
  \providecommand\transparent[1]{%
    \errmessage{(Inkscape) Transparency is used (non-zero) for the text in Inkscape, but the package 'transparent.sty' is not loaded}%
    \renewcommand\transparent[1]{}%
  }%
  \providecommand\rotatebox[2]{#2}%
  \newcommand*\fsize{\dimexpr\f@size pt\relax}%
  \newcommand*\lineheight[1]{\fontsize{\fsize}{#1\fsize}\selectfont}%
  \ifx\svgwidth\undefined%
    \setlength{\unitlength}{545.29887042bp}%
    \ifx\svgscale\undefined%
      \relax%
    \else%
      \setlength{\unitlength}{\unitlength * \real{\svgscale}}%
    \fi%
  \else%
    \setlength{\unitlength}{\svgwidth}%
  \fi%
  \global\let\svgwidth\undefined%
  \global\let\svgscale\undefined%
  \makeatother%
  \begin{picture}(1,0.53824051)%
    \lineheight{1}%
    \setlength\tabcolsep{0pt}%
    \put(0,0){\includegraphics[width=\unitlength,page=1]{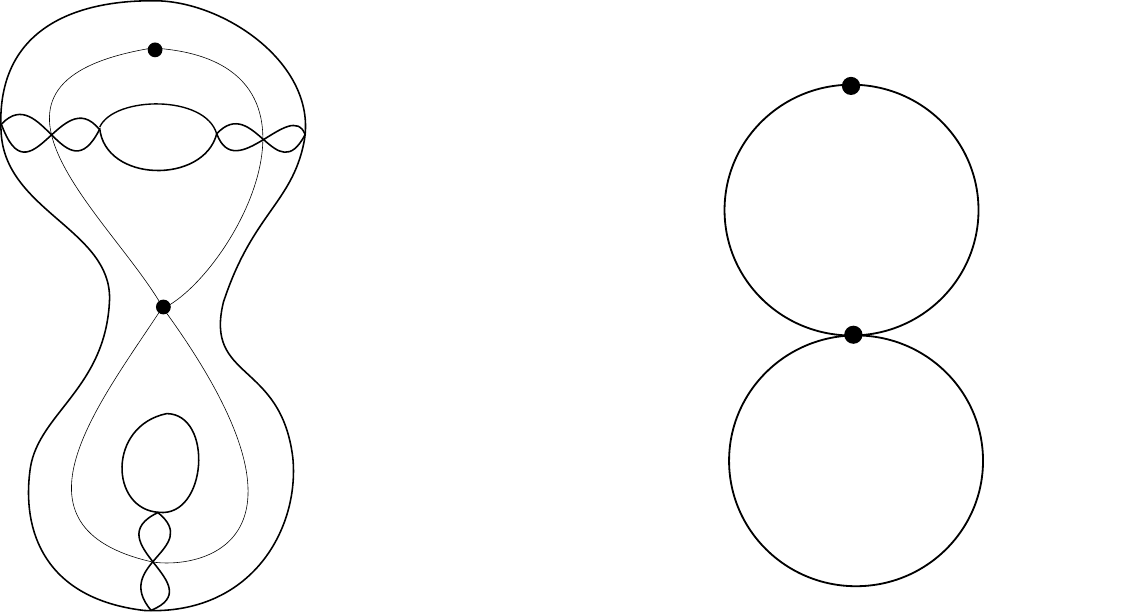}}%
    \put(0.21181046,0.23794292){\color[rgb]{0,0,0}\makebox(0,0)[lt]{\lineheight{1.25}\smash{\begin{tabular}[t]{l}$M$\end{tabular}}}}%
    \put(0.85274342,0.23381672){\color[rgb]{0,0,0}\makebox(0,0)[lt]{\lineheight{1.25}\smash{\begin{tabular}[t]{l}$G(M)$\end{tabular}}}}%
  \end{picture}%
\endgroup%

    \caption{Colored $P^2$-graph $G(M)$ for $M$ without boundary.}
    \label{fig:colored.graph.closed}
\end{figure}


\begin{figure}[htbp]
    \centering
	\def\svgwidth{.7\textwidth}
\begingroup%
  \makeatletter%
  \providecommand\color[2][]{%
    \errmessage{(Inkscape) Color is used for the text in Inkscape, but the package 'color.sty' is not loaded}%
    \renewcommand\color[2][]{}%
  }%
  \providecommand\transparent[1]{%
    \errmessage{(Inkscape) Transparency is used (non-zero) for the text in Inkscape, but the package 'transparent.sty' is not loaded}%
    \renewcommand\transparent[1]{}%
  }%
  \providecommand\rotatebox[2]{#2}%
  \newcommand*\fsize{\dimexpr\f@size pt\relax}%
  \newcommand*\lineheight[1]{\fontsize{\fsize}{#1\fsize}\selectfont}%
  \ifx\svgwidth\undefined%
    \setlength{\unitlength}{500.22533987bp}%
    \ifx\svgscale\undefined%
      \relax%
    \else%
      \setlength{\unitlength}{\unitlength * \real{\svgscale}}%
    \fi%
  \else%
    \setlength{\unitlength}{\svgwidth}%
  \fi%
  \global\let\svgwidth\undefined%
  \global\let\svgscale\undefined%
  \makeatother%
  \begin{picture}(1,0.60843221)%
    \lineheight{1}%
    \setlength\tabcolsep{0pt}%
    \put(0,0){\includegraphics[width=\unitlength,page=1]{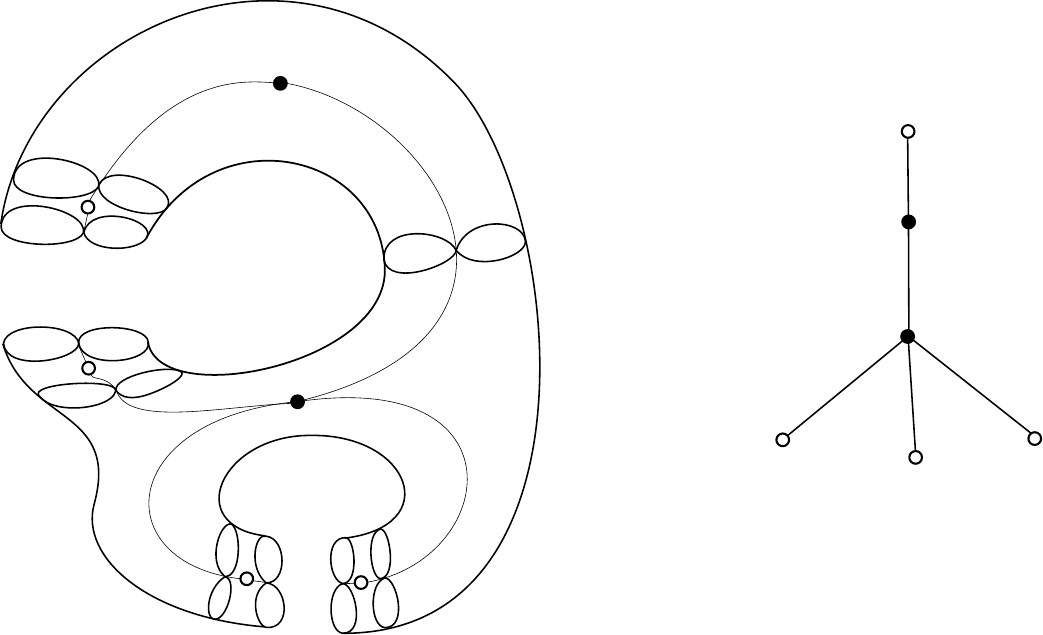}}%
    \put(0.23076812,0.33602855){\color[rgb]{0,0,0}\makebox(0,0)[lt]{\lineheight{1.25}\smash{\begin{tabular}[t]{l}$M$\end{tabular}}}}%
    \put(0.82355607,0.08375876){\color[rgb]{0,0,0}\makebox(0,0)[lt]{\lineheight{1.25}\smash{\begin{tabular}[t]{l}$G(M)$\end{tabular}}}}%
  \end{picture}%
\endgroup%

    \caption{Colored $P^2$-graph $G(M)$ for $M$ with non-empty boundary.}
    \label{fig:colored.graph.boundary}
\end{figure}

We will also use the following results from  \cite{Heil-Negami1986, negami1981}.

\begin{proposition}[\protect{cf.~\cite[Theorem~2]{negami1981}}]\label{prop:graph.even.degree.manifold}
If $G$ is a graph with all vertices of even degree, then there are infinitely many (non-homeomorphic) closed irreducible $3$-manifolds $M$ with $G(M)$ isomorphic to $G$.
\end{proposition}

\begin{proposition}[\protect{cf.~\cite[Corollary 2]{Heil-Negami1986}}]
\label{prop:isomorphic.P2.graphs}
    Let $M_1$ and $M_2$ be compact $3$-manifolds, possibly with boundary consisting only of projective planes. If  $M_1$ and $M_2$ are homeomorphic, then their $P^2$-graphs $G(M_1)$ and $G(M_2)$ are isomorphic.
\end{proposition}

With these preliminaries in hand, we are ready to prove Theorem~\ref{thm:connected.sum.depends.on.points}.
 \subsection*{Proof of Theorem~\ref{thm:connected.sum.depends.on.points}} Proposition~\ref{prop:graph.even.degree.manifold} implies that, if $G$ is a graph with all black vertices of even degree and $m$ white vertices of degree $1$, then there exist infinitely many compact irreducible $3$-manifolds $M$ with $G(M)$ isomorphic to $G$ as colored graphs and $\partial M$ consisting of $m$ projective planes (see \cite[Theorem 4]{Heil-Negami1986}). This allows us to construct an irreducible $3$-manifold $M$ with one essential separating $P^2$ and $\partial M$ four projective planes as in Figure~ \ref{construction_M}. We may now cap off $\partial M$ with cones over $P^2$ to get $Q$, a closed Alexandrov $3$-space with four singular points.

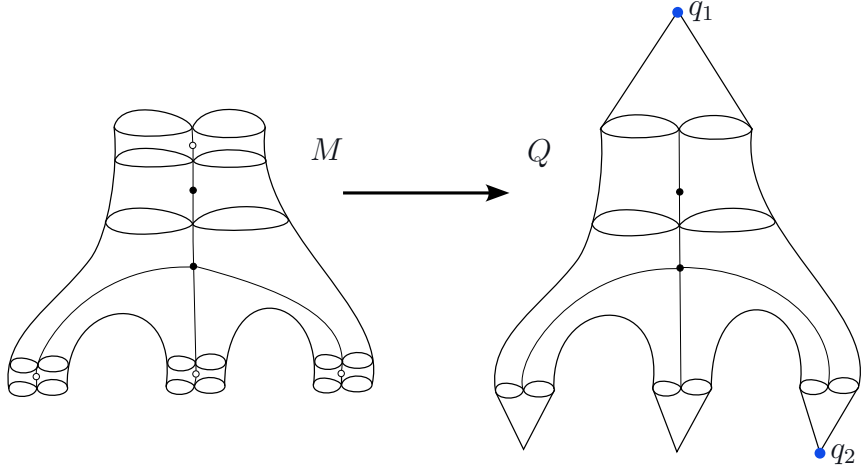
\begin{figure}
    \centering
	\def\svgwidth{.7\textwidth}
\begingroup%
  \makeatletter%
  \providecommand\color[2][]{%
    \errmessage{(Inkscape) Color is used for the text in Inkscape, but the package 'color.sty' is not loaded}%
    \renewcommand\color[2][]{}%
  }%
  \providecommand\transparent[1]{%
    \errmessage{(Inkscape) Transparency is used (non-zero) for the text in Inkscape, but the package 'transparent.sty' is not loaded}%
    \renewcommand\transparent[1]{}%
  }%
  \providecommand\rotatebox[2]{#2}%
  \newcommand*\fsize{\dimexpr\f@size pt\relax}%
  \newcommand*\lineheight[1]{\fontsize{\fsize}{#1\fsize}\selectfont}%
  \ifx\svgwidth\undefined%
    \setlength{\unitlength}{523.51587971bp}%
    \ifx\svgscale\undefined%
      \relax%
    \else%
      \setlength{\unitlength}{\unitlength * \real{\svgscale}}%
    \fi%
  \else%
    \setlength{\unitlength}{\svgwidth}%
  \fi%
  \global\let\svgwidth\undefined%
  \global\let\svgscale\undefined%
  \makeatother%
  \begin{picture}(1,0.52294873)%
    \lineheight{1}%
    \setlength\tabcolsep{0pt}%
    \put(0,0){\includegraphics[width=\unitlength,page=1]{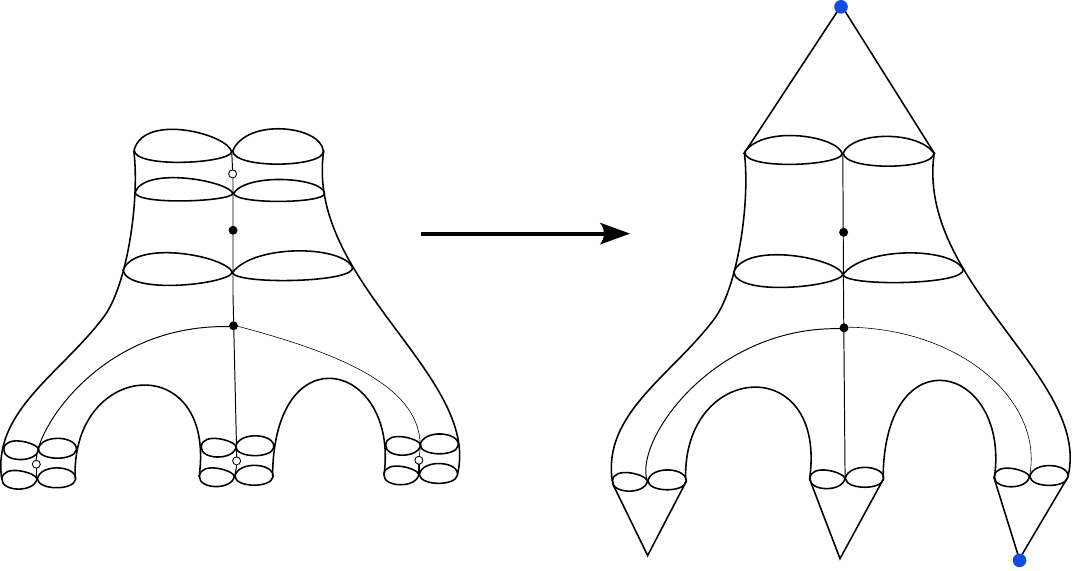}}%
    \put(0.34820903,0.34697479){\color[rgb]{0.05882353,0.0745098,0.10588235}\makebox(0,0)[lt]{\lineheight{1.25}\smash{\begin{tabular}[t]{l}$M$\end{tabular}}}}%
    \put(0.59789749,0.34520299){\color[rgb]{0.05882353,0.0745098,0.10588235}\makebox(0,0)[lt]{\lineheight{1.25}\smash{\begin{tabular}[t]{l}$Q$\end{tabular}}}}%
    \put(0.78352754,0.51342634){\color[rgb]{0.05882353,0.0745098,0.10588235}\makebox(0,0)[lt]{\lineheight{1.25}\smash{\begin{tabular}[t]{l}$q_1$\end{tabular}}}}%
    \put(0.94676362,0.00345845){\color[rgb]{0.05882353,0.0745098,0.10588235}\makebox(0,0)[lt]{\lineheight{1.25}\smash{\begin{tabular}[t]{l}$q_2$\end{tabular}}}}%
  \end{picture}%
\endgroup%

    \caption{Construction of the Alexandrov space $Q$.}
    \label{construction_M}
\end{figure}

Let $q_1$ and $q_2$ be the topologically singular points in $Q$ indicated in Figure~\ref{construction_M} and let $X_1=Q\#^{q_1,q_1} Q$ and  $X_2=Q\#^{q_2,q_2} Q$. These spaces are illustrated in Figure \ref{X1_and_X2}.

\begin{figure}
    \centering
    \def\svgwidth{1\textwidth}
\begingroup%
  \makeatletter%
  \providecommand\color[2][]{%
    \errmessage{(Inkscape) Color is used for the text in Inkscape, but the package 'color.sty' is not loaded}%
    \renewcommand\color[2][]{}%
  }%
  \providecommand\transparent[1]{%
    \errmessage{(Inkscape) Transparency is used (non-zero) for the text in Inkscape, but the package 'transparent.sty' is not loaded}%
    \renewcommand\transparent[1]{}%
  }%
  \providecommand\rotatebox[2]{#2}%
  \newcommand*\fsize{\dimexpr\f@size pt\relax}%
  \newcommand*\lineheight[1]{\fontsize{\fsize}{#1\fsize}\selectfont}%
  \ifx\svgwidth\undefined%
    \setlength{\unitlength}{540.85667407bp}%
    \ifx\svgscale\undefined%
      \relax%
    \else%
      \setlength{\unitlength}{\unitlength * \real{\svgscale}}%
    \fi%
  \else%
    \setlength{\unitlength}{\svgwidth}%
  \fi%
  \global\let\svgwidth\undefined%
  \global\let\svgscale\undefined%
  \makeatother%
  \begin{picture}(1,0.30413472)%
    \lineheight{1}%
    \setlength\tabcolsep{0pt}%
    \put(0,0){\includegraphics[width=\unitlength,page=1]{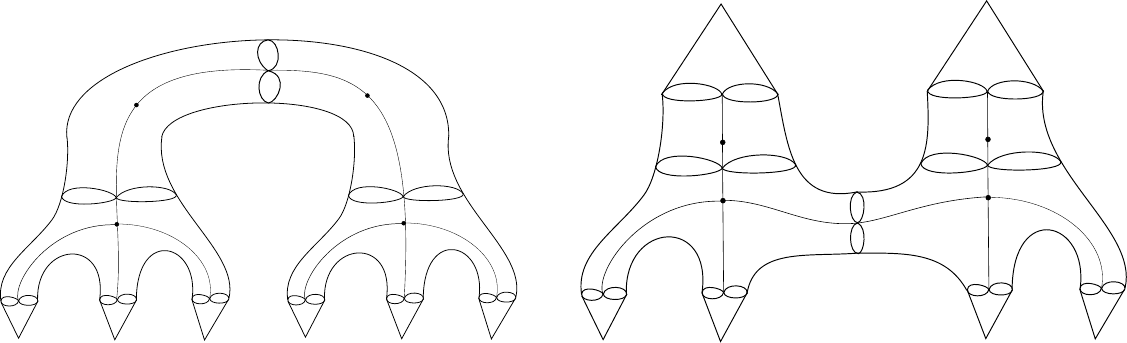}}%
    \put(0.01653684,0.22352392){\color[rgb]{0.05882353,0.0745098,0.10588235}\makebox(0,0)[lt]{\lineheight{1.25}\smash{\begin{tabular}[t]{l}$X_1$\end{tabular}}}}%
    \put(0.52450316,0.22496697){\color[rgb]{0.05882353,0.0745098,0.10588235}\makebox(0,0)[lt]{\lineheight{1.25}\smash{\begin{tabular}[t]{l}$X_2$\end{tabular}}}}%
  \end{picture}%
\endgroup%

    \caption{The two different connected sums of two copies of $Q$.}
    \label{X1_and_X2}
\end{figure}

Let us now show that $X_1$ and $X_2$ are not homeomorphic. We will proceed by contradiction.
Suppose there is a homeomorphism $h\colon X_1\to X_2$. Then $h$ sends topologically singular points to topologically singular points and, therefore, induces a homeomorphism $M_{X_1}\to M_{X_2}$, where $M_{X_i}$ is $X_i$ without open neighborhoods around the topologically singular points (see Figure \ref{M1_and_M2}). By Proposition~\ref{prop:isomorphic.P2.graphs}, $G(M_{X_1})$ and $G(M_{X_2})$ are isomorphic as colored graphs. However, this is clearly not the case as shown in Figure \ref{graphs_M1_and_M2}.
\hfill \qed

\begin{figure}
    \centering
    \def\svgwidth{1\textwidth}
\begingroup%
  \makeatletter%
  \providecommand\color[2][]{%
    \errmessage{(Inkscape) Color is used for the text in Inkscape, but the package 'color.sty' is not loaded}%
    \renewcommand\color[2][]{}%
  }%
  \providecommand\transparent[1]{%
    \errmessage{(Inkscape) Transparency is used (non-zero) for the text in Inkscape, but the package 'transparent.sty' is not loaded}%
    \renewcommand\transparent[1]{}%
  }%
  \providecommand\rotatebox[2]{#2}%
  \newcommand*\fsize{\dimexpr\f@size pt\relax}%
  \newcommand*\lineheight[1]{\fontsize{\fsize}{#1\fsize}\selectfont}%
  \ifx\svgwidth\undefined%
    \setlength{\unitlength}{540.85663082bp}%
    \ifx\svgscale\undefined%
      \relax%
    \else%
      \setlength{\unitlength}{\unitlength * \real{\svgscale}}%
    \fi%
  \else%
    \setlength{\unitlength}{\svgwidth}%
  \fi%
  \global\let\svgwidth\undefined%
  \global\let\svgscale\undefined%
  \makeatother%
  \begin{picture}(1,0.23696682)%
    \lineheight{1}%
    \setlength\tabcolsep{0pt}%
    \put(0,0){\includegraphics[width=\unitlength,page=1]{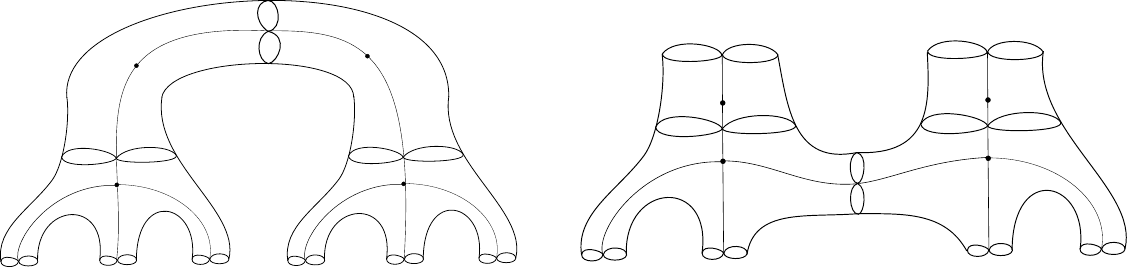}}%
    \put(0.20952944,0.1427762){\color[rgb]{0.05882353,0.0745098,0.10588235}\makebox(0,0)[lt]{\lineheight{1.25}\smash{\begin{tabular}[t]{l}$M_{X_1}$\end{tabular}}}}%
    \put(0.73661855,0.14864869){\color[rgb]{0.05882353,0.0745098,0.10588235}\makebox(0,0)[lt]{\lineheight{1.25}\smash{\begin{tabular}[t]{l}$M_{X_2}$\end{tabular}}}}%
  \end{picture}%
\endgroup%

    \caption{The $3$-manifolds $M_{X_1}$ and $M_{X_2}$ obtained from $X_1$ and $X_2$.}
    \label{M1_and_M2}
\end{figure}

\begin{figure}
	\centering
	\def\svgwidth{.9\textwidth}
\begingroup%
  \makeatletter%
  \providecommand\color[2][]{%
    \errmessage{(Inkscape) Color is used for the text in Inkscape, but the package 'color.sty' is not loaded}%
    \renewcommand\color[2][]{}%
  }%
  \providecommand\transparent[1]{%
    \errmessage{(Inkscape) Transparency is used (non-zero) for the text in Inkscape, but the package 'transparent.sty' is not loaded}%
    \renewcommand\transparent[1]{}%
  }%
  \providecommand\rotatebox[2]{#2}%
  \newcommand*\fsize{\dimexpr\f@size pt\relax}%
  \newcommand*\lineheight[1]{\fontsize{\fsize}{#1\fsize}\selectfont}%
  \ifx\svgwidth\undefined%
    \setlength{\unitlength}{535.53288978bp}%
    \ifx\svgscale\undefined%
      \relax%
    \else%
      \setlength{\unitlength}{\unitlength * \real{\svgscale}}%
    \fi%
  \else%
    \setlength{\unitlength}{\svgwidth}%
  \fi%
  \global\let\svgwidth\undefined%
  \global\let\svgscale\undefined%
  \makeatother%
  \begin{picture}(1,0.32351185)%
    \lineheight{1}%
    \setlength\tabcolsep{0pt}%
    \put(0,0){\includegraphics[width=\unitlength,page=1]{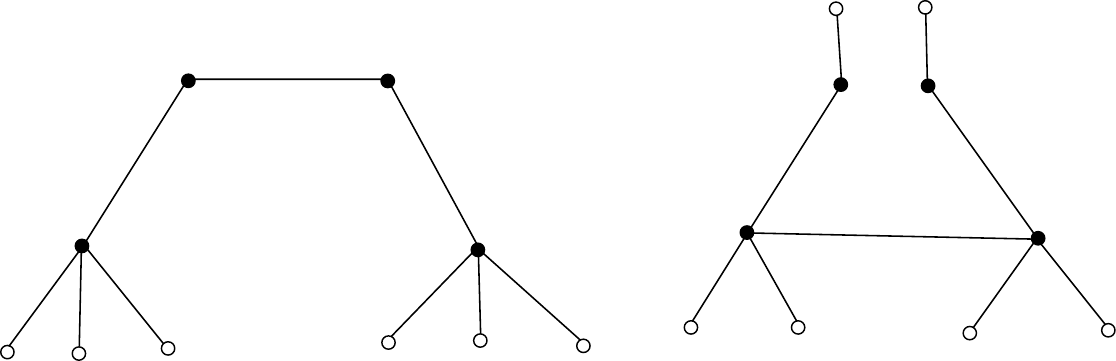}}%
    \put(0.21164451,0.15779012){\color[rgb]{0.05882353,0.0745098,0.10588235}\makebox(0,0)[lt]{\lineheight{1.25}\smash{\begin{tabular}[t]{l}$G(M_{X_1})$\end{tabular}}}}%
    \put(0.74804696,0.1563973){\color[rgb]{0.05882353,0.0745098,0.10588235}\makebox(0,0)[lt]{\lineheight{1.25}\smash{\begin{tabular}[t]{l}$G(M_{X_2})$\end{tabular}}}}%
  \end{picture}%
\endgroup%

    \caption{The colored graphs $G(M_{X_1})$ and $G(M_{X_2})$, which are not isomorphic.}
    \label{graphs_M1_and_M2}
\end{figure}

\subsection{Further definitions and basic observations.}

Let us now recall the definition of irreducibility for Alexandrov $3$-spaces (cf.\ \cite{galaz-guijarro2020}). 
Note that this definition includes topologically singular Alexandrov spaces and, if the space is a manifold, it reduces to the classical definition of irreducibility for $3$-manifolds. 


\begin{definition}[Irreducible space]\label{def:irreducibility}

A space $P\in \cA$ is \textit{irreducible} if every (PL) embedded $2$-sphere in $P$ bounds a $3$-ball and, if the set of topologically singular points of $P$ is non-empty, we further require that every two-sided (PL) $P^2$ bounds a $K(P^2)$, a cone over $P^2$.
\end{definition}

 
\begin{definition}[Boundary-parallel projective plane]
Let $X\in \cA$ be topologically singular. A projective plane $P^2 \subset M_X$
is \emph{boundary parallel} if $P^2$ is  parallel to a boundary component of $M_X$.
\end{definition}


\begin{remark}
\label{rem:push.to.manifold.part}
If $S$ is a $2$-sphere or projective plane in $X$, then, by a small deformation, we may push off $S$ from the singular points and we will assume, without loss of generality, that $S\subset M_X$. Note that a $2$-sphere $S\subset X$ bounds a ball in $X$ if and only if $S$ bounds a ball in $M_X$. This is because a $3$-ball does not contain a $P^2$, which can be seen using Alexander duality.
\end{remark}


\begin{lemma}
\label{lem:irreducibility.characterization} 
Let $X$ be a closed topologically singular  Alexandrov $3$-space. Then $X$ is irreducible if and only if $M_X$ is irreducible and every two-sided projective plane in $M_X$ is boundary-parallel. 
\end{lemma}


\begin{proof}


Suppose first that $X$ is irreducible. 
Let $S$ be an embedded $2$-sphere in $M_X\subset X$. 
Then, since $X$ is irreducible, $S$ bounds a $3$-ball in $X$. 
It follows from Remark~\ref{rem:push.to.manifold.part} that $S$ also bounds a $3$-ball in $M_X$. Hence, $M_X$ is irreducible.
 
Let now $S$ be a two-sided projective plane in $M_X\subset X$. We will show that $S$ is boundary-parallel. Let us show first that $S$ must separate  $M_X$, i.e.\ that $M_X\setminus S$ consists of two disjoint connected components $M_1$ and $M_2$.  Suppose, for the sake of contradiction, that $S$  does not separate $M_X$. Then $S$ does not separate $X$ and, therefore, $S$ does not bound a $K(P^2)$ in $X$, contradicting our assumption that $X$ is irreducible. Hence, $S$ must separate $M_X$ into two $3$-manifolds $M_1$ and $M_2$ and, therefore, $S$ must also separate $X$. Since $S$ is a two-sided projective plane and $X$ is irreducible, it follows that one of the two connected components into which $S$ separates $X$ must be a $K(P^2)$. It follows that one of $M_1$ or $M_2$ is homeomorphic to $P^2\times [0,1]$ with $S$ corresponding to one of the two boundary components. Since the other boundary component of $P^2\times [0,1]$ is, by construction, a component of $\partial M_X$, it follows that $S$ is boundary-parallel. This finishes the proof of the ``if'' part of the lemma.

  
Suppose now that $M_X$ is irreducible and every two-sided projective plane in $M_X$ is boundary-parallel. Let $S$ be an embedded $2$-sphere in $X$. Since $M_X$ is irreducible, it follows from  Remark~\ref{rem:push.to.manifold.part} that $S$ bounds a ball in $M_X$. In particular, $S$ bounds a ball in $X$. Let now $S$ be a two-sided projective plane $P^2$ in $X$. Then $P^2$ is parallel in $M_X$ to a projective plane $P'$ in $\partial M_X$. Since $P'$ bounds a cone in $X$, it follows that $P^2$ bounds a cone $K(P^2)$ in $X$.
 \end{proof}


\begin{definition}[Prime space]
\label{def:prime.space}
\label{prime}
A space $P\in \cA$ is \textit{prime} if, whenever \[P=Q_1\#^{q_1,q_2}Q_2\] for some $Q_1,Q_2\in\cA$ and points $q_1\in Q_1$, $q_2\in Q_2$, then one of the following conditions is satisfied:
\begin{enumerate}
\item If $\Sigma_{q_1}\approx \Sigma_{q_2}\approx S^2$, then either $Q_1\approx S^3$ or $Q_2\approx S^3$.
\item If $\Sigma_{q_1}\approx \Sigma_{q_2}\approx P^2$, then either $X\approx \mathrm{Susp}(P^2)$ or $Y\approx \mathrm{Susp}(P^2)$.
\end{enumerate}
We say that $P$ has a \emph{prime decomposition} if it can be written as a connected sum of prime spaces.
\end{definition}


\begin{remark}\label{prime_manifold}
Both $S^3$ and $\Susp(P^2)$ are prime.
Note that, if every \textit{separating} $2$-sphere in $P\in \cA$ bounds a ball,  and every \textit{separating} projective plane $P^2$ in $P$ bounds a cone $K(P^2)$ in $P$, then $P$ is prime.
\end{remark}


Let us now define a \textit{normal prime decomposition}, a special type of prime decomposition for closed Alexandrov $3$-spaces (cf.~\cite{heil} for the manifold definition).

\begin{definition}[Normal prime decomposition]
\label{def:normal.prime.decomposition}
    Let $X\in\mathcal{A}$. Then $M_X$ is a non-orientable $3$-manifold with boundary a collection of projective planes, and $M_X$ has a unique prime factorization with respect to the manifold connected sum $\#$ given by
    \[
    M_X =M_1 \# \cdots \# M_n \# l(S^1 \tilde{\times}S^2 ),
    \]
    where $l(S^1\tilde{\times}S^2)$ denotes the connected sum of $l\geq 0$ non-orientable $S^2$-bundles over $S^1$. A \emph{normal prime decomposition} of $X$ is a prime decomposition given by
    \[X=(P_{11}\hat{\#} \cdots \hat{\#} P_{1k_1})\# \cdots \#(P_{n1}\hat{\#} \cdots \hat{\#} P_{nk_n}) \# l(S^1 \tilde{\times}S^2 )\]
    where $\widehat{M}_i =P_{i1}\hat{\#} \cdots \hat{\#} P_{ik_i}$ is a prime factorization with respect to the non-manifold connected sum $\hat{\#}$ of $\widehat{M_i }$, the Alexandrov $3$-space obtained by capping off the $P^2$ boundary components  of $M_i$.
\end{definition}

\begin{remark}
When $X\in \mathcal{A}$ is a manifold, then $X$ admits a unique normal prime decomposition (see \cite[Proposition on p.~143]{heil}).
\end{remark}

The following proposition extends a basic result for $3$-manifolds to general closed Alexandrov $3$-spaces (see, for example, \cite[Ch.~3]{hempel} or \cite[Lemma 1]{milnor}).


\begin{proposition}
\label{prop:irreducible.then.prime}
If $P\in \cA$ is irreducible, then $P$ is prime.
\end{proposition}


\begin{proof}
If $P$ is a manifold, then the assertion is well-known and the proposition follows from the fact that our definition of irreducibility coincides with that of irreducibility for manifolds. Assume then that $P$ is not a manifold and suppose that there exist $Q_1,Q_2\in \cA$ and points $q_1\in Q_1$, $q_2\in Q_2$ such that $P=Q_1\#^{q_1,q_2}Q_2$. Then, $\partial(Q_1\setminus U_{q_1}) = \partial(Q_2\setminus V_{q_2})$ is a two-sided embedded sphere or projective plane in $P$, where $U_{q_1}$ and $V_{q_2}$ are sufficiently small open neighborhoods of $q_1$ and $q_2$ homeomorphic to $K_{q_1}$ and $K_{q_2}$, respectively. Hence, since $P$ is irreducible, one of $Q_1\setminus U_{q_1}$ or $Q_2\setminus V_{q_2}$ must be homeomorphic to a $3$-ball or a $K(P^2)$. In particular, one of $Q_1$ or $Q_2$ must be $S^3$ or $\Susp(P^2)$ and hence $P$ must be prime. 
\end{proof}

We now prove two results on prime spaces that we will use further below.


\begin{lemma}
\label{lem:prime.iff.mx.prime.separating.p2.is.boundary.parallel.heil.3.2.3}
Let $X$ be a closed topologically singular  Alexandrov  $3$-space. Then $X$ is prime if and only if $M_X$ is prime and every separating $P^2$ in $M_X$ is boundary-parallel.
\end{lemma}

\begin{proof}
Suppose $X$ is prime. If there exist 3-manifolds $N_1$ and $N_2$ such that $M_X = N_1\# N_2$, then the 2-sphere separating $M_X$ separates $X$ into $X_1\# X_2$ with $N_i = M_{X_i}, i = 1,2$. It follows that $X_1$ or $X_2$ is $S^3$ and $M_{X_1}$ or $M_{X_2}=S^3$. If there exists a separating projective plane $P^2$ in $M_X$, then $P^2$ bounds  $K(P^2)$ in $X$ and it follows that $P^2$ is boundary-parallel in $M_X$.
 
 Now suppose $M_X$ is irreducible and every two-sided projective plane in $M_X$ is
boundary-parallel. If  there is a sphere $S^2$ that separates $X$, then $S^2$  separates $M_X$. Since $M_X$ is irreducible, $S^2$ bounds a ball in $M_X$ and hence in $X$. If 
 there is a projective plane $P^2$ in $X$ that decomposes $X$ as $X_1 {\hat\#}X_2$, then $P^2$ is parallel to a boundary component of $\partial M_X$ and hence bounds a $K(P^2)$ in $X$, i.e.\ $X_1$ or $X_2 \approx \Susp(P^2)$.
\end{proof}


\begin{proposition}
\label{prop:prime.not.irreducible}
If $P\in \cA$ is prime and not irreducible, then either $P$ is homeomorphic to $S^2{\times}S^1$ or  $S^2{\tilde\times}S^1$ (the non-orientable $S^2$-bundle over $S^1$), or $P$ is not a manifold and contains a non-separating $P^2$.
\end{proposition}

\begin{proof}
If $P$ is a manifold, then, since our notion of irreducibility coincides with irreducibility for $3$-manifolds, it is well-known that $P$ must be homeomorphic to $S^2\times S^1$ or $S^2\tilde{\times} S^1$. 

Suppose now that $P$ is not a manifold. Since $P$ is not irreducible, there exists an embedded surface $S\subset X$ which is either an $S^2$ which does not bound a $3$-ball or a two-sided $P^2$ in which does not bound a $K(P^2)$. Suppose first that $S$ is an $S^2$. Then, by Remark~\ref{rem:push.to.manifold.part}, the $2$-sphere $S$ cannot bound a ball in $M_P$ and, by Lemma~\ref{lem:prime.iff.mx.prime.separating.p2.is.boundary.parallel.heil.3.2.3}, $M_P$ is prime. Since $M_P$ is prime and $S$ does not bound a $3$-ball, $S$ cannot separate $M_P$. Moreover, $S$ is not boundary-parallel, since every boundary component of $M_P$ is homeomorphic to $P^2$. Then, by a standard argument in $3$-manifold topology, $M_P$ splits as a connected sum of some $3$-manifold $N_P$ and $S^2\times S^1$ or $S^2\tilde{\times} S^1$. Thus, $M_P$, and, in turn, $P$, is not prime, which is a contradiction. Therefore, $S$ must be a two-sided $P^2$ that does not bound a $K(P^2)$ in $P$. Therefore, $S$ is not boundary-parallel in $M_P$. Thus, by Lemma~\ref{lem:prime.iff.mx.prime.separating.p2.is.boundary.parallel.heil.3.2.3}, $S$ is non-separating.
\end{proof}


\subsection{Double branched covers and irreducibility} 
We now discuss some relations between the irreducibility of a topologically singular closed Alexandrov $3$-space $X$ and that of its double branched cover $\widetilde{X}$. Let us start with the following observations. 


\begin{lemma}
\label{lem:involution.3.ball}
Let $\iota\colon B^3\to B^3$ be a PL involution on the $3$-ball that restricts to the antipodal map on $S^2=\partial B^3$. Then $\iota$ is the cone over the antipodal map on $S^2$.
\end{lemma}


\begin{proof}
Let $S^3=B^3\cup_{S^2}B^3$ and define $h\colon S^3\to S^3$ as $h(x)=\iota(x)$. This is an involution of the 3-sphere to itself which reverses orientation. Thus, by Smith theory, the fixed point set $F$ of $h$ must be a homology $r$-sphere for $0\leq r\leq 3$ and, as the involution reverses orientation, $F$ must be a $2$-sphere or two points. We already know that $F$ cannot be a wildly embedded $2$-sphere, since $\iota$ is a PL involution. If $F$ were a tame $2$-sphere, then, by Hirsch--Smale \cite{hirsch-smale1959}, $h$ would be equivalent to a reflection through an equatorial $2$-sphere of $S^3$. However, this does not happen as this would fix the equatorial $2$-sphere where we actually know the involution is the antipodal map. Thus, the set of fixed points of $h$ must be two points and, by Hirsch--Smale--Livesay (see \cite[Theorem 1.1]{hirsch-smale1959} and \cite{livesay1963}), $h$ is equivalent to the involution $L\colon S^3\to S^3$ given by
\[L(x_1,x_2,x_3,x_4)=(x_1,-x_2,-x_3,-x_4).\]
Thus, the involution $h$ fixes the poles of $S^3$ and acts on the rest of the $3$-sphere $(-1,1)\times S^2$ as $h(t,x)=(t,-x)=i(t,x)$ where $-x$ is the antipodal map. 
\end{proof}


\begin{proposition}
\label{irreducible_cover}
If $X\in \cA$ is topologically singular, then its double branched cover $\widetilde{X}$ is irreducible if and only if $X$ is irreducible.
\end{proposition}

\begin{proof}
Suppose $\widetilde{X}$ is irreducible. If $S^2\subset X$ is a $2$-sphere in $X$, then $p^{-1}(S^2)=S^2_a\cup S^2_b$ are two disjoint spheres. Since $\widetilde{X}$ is irreducible, each  $2$-sphere $S^2_a$, $S^2_b$ bounds a $3$-ball $B_a$, $B_b$, respectively. Then, we can write
\begin{equation*}
    \widetilde{X}=B_a\cup B_b\cup \widetilde{X}',    
\end{equation*}
where $\widetilde{X}'=\widetilde{X}\setminus(B_a\cup B_b)$. Moreover, the involution $\iota\colon\widetilde{X}\to\widetilde{X}$ is such that $\iota(B_a)=B_b$. Indeed, $\iota$ restricts to a homeomorphism from $\widetilde{X}\setminus S^2_a$ to $\widetilde{X}\setminus S^2_b$. Now, $\mathring{B_a}$ is a connected component of $\widetilde{X}\setminus S^2_a$, so we must have either  $\iota(\mathring{B_a})=(\widetilde{X'}\cup B_a)\setminus S^2_b$ or $\iota(\mathring{B_a})=\mathring{B_b}$. In the first case, we would have, in particular, that $B_a\subseteq\iota(\mathring{B_a})\subseteq\iota(B_a)$. This implies that $\iota(B_a)\subseteq B_a$, meaning $B_a=\iota(B_a)$, which cannot happen since $S^2_a\cap \iota(S^2_a)=\emptyset$. Thus, we must have $\iota(\mathring{B_a})=\mathring{B_b}$.
Therefore,
\begin{equation*}
    X=\widetilde{X}'/\iota\cup B^3,
\end{equation*}
where $B^3=p(B_a)=p(B_b)$ and we have $\partial B^3=S^2$, our original $2$-sphere in $X$. Then, $S^2$ bounds a $3$-ball in $X$.

If $P^2\subset X$ is a two-sided projective plane in $X$ then $p^{-1}(P^2)=S^2$ is a $2$-sphere which must bound a $3$-ball $B^3$ in $\widetilde{X}$ since the latter is irreducible. We have the involution $\iota\colon\widetilde{X}\to\widetilde{X}$. Let us see that $
\iota(B^3)=B^3$. We have
\begin{equation*}
    \widetilde{X}=\widetilde{X}'\cup_{S^2} B^3,
\end{equation*}
where $\widetilde{X}'=\overline{\widetilde{X}\setminus B^3}$ and
\begin{equation*}
    \iota\colon \widetilde{X}'\cup_{S^2} B^3\to\widetilde{X}'\cup_{S^2} B^3.
\end{equation*}
We have two cases: $\widetilde{X}'=\iota(B^3)$ or $B^3=\iota(B^3)$. Suppose first that $\widetilde{X}'=\iota(B^3)$. Since $\iota$ is a homeomorphism, $\widetilde{X}'=\iota(B^3)$ is another $3$-ball and it follows that $\widetilde{X} = \widetilde{X}'\cup_{S^2} B^3$ is a $3$-sphere. Moreover, $\iota\colon \widetilde{X}\to \widetilde{X}$  exchanges the two $3$-balls and restricts to the antipodal map on the equatorial $S^2$ corresponding to their common boundary. 

Then the projective plane $P^2\subset X$ is one-sided, which is a contradiction, since we have assumed that $P^2$ is two-sided.
Therefore, $B^3=\iota(B^3)$. Hence, by Lemma~\ref{lem:involution.3.ball}, on $B^3$ the involution $\iota$ is the cone over the antipodal map on the boundary sphere of $B^3$ and $B^3/\iota = K(P^2)$, the cone over $P^2$. 
This shows that
\begin{align*}
    X=\widetilde{X}'/\iota\cup_{P^2}K(P^2).
\end{align*} 
Thus, $X$ is irreducible.

We now prove that the irreducibility of $X$ implies that of $\widetilde{X}$. Suppose then that $X$ is irreducible. If there is an essential $2$-sphere $S^2$ in $\tilde{X}$, we may assume that $S^2$ is disjoint from the fixed points of the involution $\iota\colon \tilde{X}\to \tilde{X}$. Then, by \cite[Lemma 1]{Tollefson73}, we may assume that $S^2$ is such that $\iota(S^2) = S^2$ or $S^2\cap \iota(S^2) = \emptyset$.
If $S^2\cap \iota(S^2)=\emptyset$, then $p(S^2)$ is a $2$-sphere in $X$ that bounds a ball $B^3\subset X$. Then a lift of $B^3$ is a ball in $\tilde{X}$ bounded by $S^2$.
If $\iota(S^2)=S^2$ then $p(S^2)$ is a projective plane in $M_X \subset X$. Since $X$ is irreducible, $p(S^2)$ bounds a $K(P^2 )$ in $X$. Then the lift of $K(P^2 )$ is a ball $B^3$ in $\tilde{X}$ bounded by $S^2$. In either case, we get a contradiction to $S^2$ being essential.
\end{proof}

We conclude this subsection with a sufficient condition for a prime topologically singular Alexandrov $3$-space to be irreducible. 


\begin{proposition}
Let $X\in \cA$ be topologically singular. If $X$ is prime and $\widetilde{X}$ has no $S^2\times S^1$ summands, then $X$ is irreducible.
\end{proposition}


\begin{proof}
 If $X$ is not irreducible, then by Proposition~\ref{prop:prime.not.irreducible}, $X$ contains a non-separating projective plane $P^2$. Now, $P^2$ lifts to a non-separating sphere in ${\tilde X}$, which implies that ${\tilde X}$ has an $S^2{\times}S^1$ summand and hence is not prime. 
\end{proof}


\subsection{Proof of Theorem~\ref{thm:existence.prime.decomposition}: Existence of a prime decomposition.}
\label{ss:proof.existence.prime.decomposition}

Let us now show that every closed Alexandrov $3$-space $X$ has a prime decomposition. If $X$ is a $3$-manifold, then the statement is the usual prime decomposition theorem for $3$-manifolds (see \cite{kneser}, \cite[Theorem 1]{milnor}, \cite[Theorem 3.15]{hempel} and \cite[p.\ 8]{hatcher}).

Suppose now that $X$ is not a $3$-manifold. Recall that $M_X$ is homeomorphic to a compact non-orientable $3$-manifold with a finite even number of  $P^2$ boundary components. Then there is a prime decomposition of $M_X$ into prime $3$-manifolds with respect to to the usual connected sum (see   \cite[Remark 1 on p.\ 5]{milnor} or \cite[Section 5]{heil}). Hence, we may write
\begin{align}
\label{eq:prime.decomposition.M_x}
M_X=M_1\#\cdots\# M_{n},
\end{align}
where each $M_i$ is irreducible or an $S^2$-bundle over $S^1$. Note that the Loop Theorem  implies that any two-sided projective plane in $M_i$ is incompressible in $M_i$ (see \cite[Section 2]{stallings} or \cite[Theorem 4.2]{hempel}).

By Haken's finiteness theorem (see \cite[Proposition 1.7]{hatcher} and \cite{haken}), for every $M_i$ different from $S^2\times S^1$ or $S^2\tilde{\times}S^1$ there is a system $\cP_i=\{P^i_1,\ldots,P^i_{r_i}\}$ consisting of a finite number of projective planes  such that any other embedded two-sided projective plane in $M_i$ is parallel to one of the $P^i_j$ and no connected component of $M_i\setminus\cP_i$ is a product $P^2\times I$. This means that every embedded two-sided projective plane in $M_i$ is either parallel to one of the $P_j^i$ or it is boundary-parallel.

After capping off the $P^2$ boundary components of each $M_i$ to get a closed Alexandrov $3$-space $\widehat{M}_i$, we have, from \eqref{eq:prime.decomposition.M_x}, that
\[
    X=\widehat{M}_1\#\cdots\#\widehat{M}_{n}.
\]
For each $i\in\{1,...,n\}$, let $\cP'_i$ be the projective planes in $\cP_i$ that are separating. Suppose that $|\cP_i'|=s_i$. Then, we can cut and capp off $\widehat{M}_i$ through every projective plane in $\cP_i'$ and get that
\[
\widehat{M_i}=B_i^1\SCS\cdots\SCS B_i^{s_i},
\]
where $B^i_j$ is an Alexandrov $3$-space for each $j\in\{1,...,s_i\}$. Note that the $B_j^i$ do not need to be irreducible, as they might have two-sided non separating projective planes. However, each $B_j^i$ is prime as there are no separating spheres or projective planes which can realize a  connected sum decomposition. Therefore
\[
X=(B^1_1\widehat{\#}\cdots\widehat{\#} B_1^{s_1})\#\cdots\#(B_{n+1}^1\SCS\cdots\SCS B_{n+1}^{s_{n+1}}),
\]
which is a connected sum prime decomposition of $X$.
\qed


\subsection{Proof of Theorem~\ref{thm:normal.decomposition.uniqueness}: Uniqueness of a normal prime decomposition.} We will now show that every closed Alexandrov $3$-space admits a normal prime decomposition. Let $X\in \mathcal{A}$. If $X$ is a $3$-manifold, then the assertion follows from a generalization of Milnor's proof of the uniqueness of a prime decomposition for $3$-manifolds (see \cite[Proposition on p. 143]{heil}). Now, suppose $X$ has topologically singular points. By \cite[Proposition on p. 143]{heil}, $M_X$ has a unique prime decomposition as a $3$-manifold given by
\[M_X=M_1\#\cdots\# M_n\#l(S^1\tilde{\times}S^2).\]
Thus, it suffices to prove Theorem~\ref{thm:normal.decomposition.uniqueness} for any closed Alexandrov $3$-space $X$ such that $M_X$ is an irreducible $3$-manifold, i.e., every $2$-sphere in $M_X$ bounds a $3$-ball. 
Suppose then that $X$ has two prime decompositions   
$P_1 \hat{\#} \cdots \hat{\#} P_m$
and 
$Q_1 \hat{\#} \cdots \hat{\#} Q_n$
with $m\leq n$. There are systems $S=\{S_1 ,\dots, S_k\}$ and $T=\{T_1 ,\dots, T_l\}$ of separating, mutually non-parallel and non-boundary parallel projective planes in $M_X$ such that $M_X\backslash S=\{M_{P_1} ,\dots,M_{P_m}\}$ and $M_X \backslash T=\{M_{Q_1} ,\dots,M_{Q_n}\}$, where $P_i$ is obtained from $M_{P_i}$ by capping off the boundary components with cones over $P^2$. Similarly, we obtain $Q_j$ by capping off $M_{Q_j}$.  By Lemma ~\ref{lem:prime.iff.mx.prime.separating.p2.is.boundary.parallel.heil.3.2.3}, since $P_i$ is prime, every separating projective plane in $M_{P_i}$ is boundary parallel in $M_{P_i}$ and every $2$-sphere in $M_{P_i}$ bounds a ball in $M_{P_i}$. Similarly, since $Q_j$ is prime, every separating projective plane in $M_{Q_j}$ is boundary parallel in $M_{Q_j}$ and every $2$-sphere in $M_{Q_j}$ bounds a ball in $M_{Q_j}$.  

By  \cite[Lemma 1.2]{negami1981} there is an isotopy of $M_X$ that carries $S=\{S_1 ,\dots, S_k\}$ into a system disjoint from $T=\{T_1 ,\dots, T_l\}$. Thus, we may assume that $S\cap T=\emptyset$. Then each $T_j$ is in some $M_{P_{i_j}}$ and is boundary parallel in $M_{P_{i_j}}$. Since $T_j$ is not parallel to a boundary of $M_X$, there is a component $S_i\in S$ of $\partial M_{P_{i_j}}$ such that $T_j$ and $S_i$ bound a submanifold $E_{ij}$ of $M_{P_{i_j}}$ that is homeomorphic to $P^2{\times}I$. 
No other $T_r$ lies in $E_{ij}$. Otherwise, $T_r$ would be parallel to $T_j$ (see, for example, \cite[ Lemma 1.1]{negami1981}). Hence, we may deform the system $T=\{T_1 ,\dots, T_l\}$ to the system
$S=\{S_1 ,\dots, S_k\}$. It follows that $m=n$ and $\{P_1 ,\dots,P_m\}$ is a permutation of $\{Q_1 ,\dots,Q_n\}$.
\qed


\section{An infinite family of prime Alexandrov \texorpdfstring{$3$}{}-spaces which are not irreducible.}
\label{s:infinitely.many.prime.non.irreducible.spaces}

In this section, we construct an infinite family of closed topologically singular Alexandrov $3$-spaces which are prime and not irreducible, proving Theorem~\ref{thm:infinite.family.prime.not.irreducible}. The existence of such a family stands in contrast to the manifold case, where every prime closed $3$-manifold is irreducible, except for $S^2\times S^1$ and $S^2{\tilde{\times}}S^1$.

Let  $F_g$ be a closed, connected, orientable surface of genus $g\geq 1$ and let $M=F_g\times S^1$. Note that the universal cover of $M$ is  $\mathbb{R}^3$, which is irreducible. Therefore, $M$ is also irreducible by \cite[Proposition 1.6]{hatcher}.

Let $\alpha$ be an orientation-reversing involution of $M$ with only isolated fixed points and let $X=M/\alpha$. Note that such an involution always exists. Indeed, since $g\geq 1$, by \cite[p.\ 49]{farb_margalit}, there is a hyperelliptic involution $\varphi$ on $F_g$ (i.e.\ an involution whose quotient space is $S^2)$. Taking the product of $\varphi$ with the conjugation $z\mapsto \bar{z}$ in $S^1\subset \mathbb{C}$ yields an $\alpha$ on $M$ with only isolated fixed points  (see, for example, \cite[p.\ 5571]{galaz-guijarro2015}). Note that $\alpha$ has at least four isolated fixed points. Thus, $X$ has at least four topologically singular points. Since $M$ is irreducible, Lemma~\ref{irreducible_cover} implies that $X$ is irreducible.
We now cut off the cones over the projective planes corresponding to sufficiently small open neighborhoods of two singular points of $M$ to get a topological space $X_0$ with two $P^2$ boundary components. After identifying the two $P^2$ boundary components of $X_0$, we obtain a closed topologically singular Alexandrov $3$-space, which we will denote by $X_g$.


\begin{theorem}
The Alexandrov $3$-space $X_g$ is prime and is not irreducible.
\end{theorem}

\begin{proof}

Note that $X_g$ is not irreducible, as the projective plane that results from identifying the boundary components in $X_0$ is two-sided and non-separating.

Let $M_X$ and $M_{X_g}$ be, respectively, the non-orientable $3$-manifolds with boundary obtained from $X$ and $X_g$ by removing sufficiently small open neighborhoods of the topologically singular points.

Since $X$ is irreducible, Lemma~\ref{lem:irreducibility.characterization} implies that $M_X$ is irreducible and every two-sided projective plane in $M_X$ is boundary-parallel. In particular, every separating projective plane in $M_X$ is boundary-parallel. Now, $M_{X_g}$ is irreducible, since it is obtained from $M_X$ by identifying two incompressible surfaces $P_0$ and $P_1$ in $\partial M_X$.
Hence, $M_{X_g}$ is prime. Thus, by Lemma~\ref{lem:prime.iff.mx.prime.separating.p2.is.boundary.parallel.heil.3.2.3}, it suffices to show that every separating projective plane in $M_{X_g}$ is boundary-parallel.

Let $P$ be a separating projective plane in $M_{X_g}=M_X \cup P^2{\times}[0,1]$, where $P_0=P^2{\times}\{0\}$ and $P_1=P^2{\times}\{1\}$. Suppose first that $P\cap (P_0 \cup P_1 )=\emptyset$. Then $P$ deforms into $M_X$ and is boundary-parallel in $M_X$. Since $P$ is separating, it is not parallel to $P_0$ nor $P_1$ and is therefore boundary-parallel in $M_{X_g}$.
Suppose now that $P\cap (P_0 \cup P_1 )\neq\emptyset$. We will show that $P$ is isotopic to a projective plane that misses $P_0 \cup P_1$. Deform $P$ so that $P\cap (P_0 \cup P_1 )$ consists of a minimal collection of simple closed curves. These curves are two-sided in $P$, $P_0$, $P_1$ and therefore bound disks in each. 
Let $c$ be an innermost intersection curve in $P$, i.e. $c$ bounds a disk $D$ in $P$ such that $D\cap (P_0 \cup P_1 )=c$. Let $D'$ be the disks bounded by $c$ in $P_0$, say. Then $D\cup D'$ (slightly deformed) is a $2$-sphere in $M_X$ that bounds a ball in $M_X$. Then $D$ can be deformed in this ball to $D'$ and then slightly off $D'$ to eliminate $c$. This gives a deformation of $P$ with fewer intersection curves with $P_0 \cup P_1 $. By minimality, $P$ misses $P_0 \cup P_1 $. Thus, by the previous case, $P$ must be boundary-parallel.
\end{proof}


\section{Proof of Mitsuishi's and Yamaguchi's conjecture}
\label{s:mitsuishi.yamaguchi}

In this section, we prove Theorem~\ref{thm:mitsuishi.yamaguchi}, which verifies  Mitsuishi's and Yamaguchi's gluing conjecture \cite[Conjecture 1.10]{mitsuishi-yamaguchi2015}. To prove this theorem, we will require the following lemmas.


\begin{lemma}\label{double_bpt}
The double of $B(\mathrm{pt})$ is homeomorphic to $\mathrm{Susp}(P^2)\#\mathrm{Susp}(P^2)$.
\end{lemma}


\begin{proof}
Let $D^3\subset M_{\mathrm{Susp}(P^2)}$ be a $3$-ball in the manifold part of $\Susp(P^2)$. Then, by \cite[Remark 2.62]{mitsuishi-yamaguchi2015},
\begin{equation*}
    K(P^2)\cup_{\mathrm{M\Ddot{o}}}K(P^2)=\overline{\mathrm{Susp}(P^2)\setminus D^3}\cong S^2\times[-1,1]/_{(\sigma,-\mathrm{id})},
\end{equation*}
where $\sigma$ is topologically conjugate to the suspension of the antipodal map on $S^1$. Then
\[
\mathrm{Susp}(P^2)\#\mathrm{Susp}(P^2)=\left( K(P^2)\cup_{\mathrm{M\Ddot{o}}}K(P^2)\right)\cup_{S^2}\left( K(P^2)\cup_{\mathrm{M\Ddot{o}}}K(P^2)\right).
\]
In each $K(P^2)\cup_{\mathrm{M\Ddot{o}}}K(P^2)$ there is a two-sided separating Möbius band, i.e. a Möbius band whose normal neighborhood is homeomorphic to $\mathrm{M\Ddot{o}}\times[-1,1]$, its boundary is the sphere $S^2$ along which we are gluing. Then, when gluing along $S^2$, we glue $\mathrm{M\Ddot{o}}\times[-1,1]\cup_{S^2}\mathrm{M\Ddot{o}}\times[-1,1]$, which is homeomorphic to $\mathrm{Kl}\times[-1,1]$. This means that there is a two-sided separating Klein bottle in $\mathrm{Susp}(P^2)\#\mathrm{Susp}(P^2)$. Therefore, because we know that $P^2=\mathrm{M\Ddot{o}}\cup_{S^1} D^2$, when splitting along this Klein bottle, we get two copies of 
\[
    K(P^2)\cup_{D^2}K(P^2)=B(\mathrm{pt}),
\]
where the last equality follows from the work of Mitsuishu and Yamaguchi (see lines before \cite[Lemma 2.61]{mitsuishi-yamaguchi2015}). Therefore, $\mathrm{Susp}(P^2)\#\mathrm{Susp}(P^2)$ is the double of $B(\mathrm{pt})$, as we wanted.
\end{proof}


\begin{lemma}\label{double_bs2}
The double of $B(S_2)$ is homeomorphic to $\mathrm{Susp}(P^2)\#\mathrm{Susp}(P^2)$.
\end{lemma}


\begin{proof}
We know from Definition \ref{spaces} that $B(S_2)$ is homeomorphic to $\mathrm{Susp}(P^2)\setminus\mathrm{int}(D^3)$, which has boundary $\s^2$. Therefore, the double of $B(S_2)$ is homeomorphic to
\[\left(\mathrm{Susp}(P^2)\setminus\mathrm{int}(D^3)\right)\cup_{\s^2}\left(\mathrm{Susp}(P^2)\setminus\mathrm{int}(D^3)\right)\]
which is $\mathrm{Susp}(P^2)\#\mathrm{Susp}(P^2)$.
\end{proof}


\begin{remark}\label{quadripus_BS4}
    The space $B(S_4)$ is the \textit{quadripus} capped with cones over the projective plane. 
     The quadripus, as defined in \cite[Example 2]{heil-larranaga}, is the punctured quotient of $T^2\times [0,1]$ via the involution $\tau(z_1,z_2,t)=(\overline{z_1},\overline{z_2},{\color{blue}1-t})$. It is easy to see that $\tau$ is topologically conjugate to the involution $f=(\sigma,-id)\colon T^2\times[-1,1]\to[-1,1]$ given by $f(z_1,z_2,t)=(\overline{z_1},\overline{z_2},-t)$, and we have used the latter involution to define the quadripus in Section~\ref{ss:special.manifolds}. Thus, we have that $T^2\times [0,1]/\tau\approx T^2\times[-1,1]/(\sigma,-id)=B(S_4)$ (see Definition~\ref{spaces}). Moreover, if $C_i$ is an invariant $3$-ball neighborhood of a fixed point of $\tau$, then $C_i/\tau$ is a cone over the projective plane. Thus, $T^2\times [0,1]/\tau \approx B(S_4)$ is just the \textit{quadripus} capped with cones over the projective plane.
\end{remark}


\begin{remark}\label{octopod_quadripus}
The \textit{octopod} is homeomorphic to $Q\cup_{T^2}Q$ (see \cite[Example 4, Section 3]{heil-larranaga}).
\end{remark}


\begin{lemma}\label{double_bs4}
The double of $B(S_4)$ is the quotient of the $3$-torus $T^3$ via the involution $\beta\colon T^3\to T^3$ given by $\beta(z_1,z_2,z_3)=(\overline{z_1},\overline{z_2},\overline{z_3})$. 
\end{lemma}


\begin{proof}
The double of $B(S_4)$ is $B(S_4)\cup_{T^2}B(S_4)$ which, by Remark \ref{quadripus_BS4}, is the same as capping off $Q\cup_{\t^2}Q$ with cones over the projective plane. This is the same as capping of the \textit{octopod}, which results in the quotient of $T^3$ via the involution $\beta$ (see Definition \ref{ss:special.manifolds}).
\end{proof}


\begin{lemma}\label{bipod_tetrapod}
The \textit{bipod} and \textit{tetrapod} capped with cones over the projective plane are not simply-connected.
\end{lemma}


\begin{proof}
Let $I$ be a closed interval.
From \cite[Example 5]{heil-larranaga}, it follows that the \textit{bipod}, $B$, may be viewed as $D\cup(\Kl\tilde{\times}I)$, where $\Kl\tilde{\times} I$ is the non-orientable $I$-bundle over the Klein bottle, $\Kl$, {\color{blue} and $D$ is the dipus, defined in Section~\ref{ss:special.manifolds}}. In this decomposition,  $D\cap(\Kl\tilde{\times}I)=\partial(\Kl\tilde{\times}I)=\partial_K D$, the Klein bottle boundary component of $D$. Also, the \textit{tetrapod}, $\TP$, may be viewed as $Q\cup_{T_0}(T^2\tilde{\times}I)$, where $T^2\tilde{\times}I$ is the non-orientable twisted $I$-bundle over the $2$-torus and $T_0$ is the torus boundary of the \textit{quadripus}.

Let $\widehat{B}$, $\widehat{D}$, $\widehat{Q}$ and $\widehat{\TP}$ be the \textit{bipod}, \textit{dipus}, \textit{quadripus} and \textit{tetrapod}, respectively, capped with cones over the projective plane. Let us first make the following observations:

\begin{enumerate}[label=(\roman*)]
    \item If $X$ is a simply-connected space, then $X$ does not admit a non-trivial covering $p\colon\tilde{X}\to X$.\\
    \item Let $K_1=\partial(\Kl\tilde{\times}I)$. Then, the $2$-sheeted covering space of $\Kl\tilde{\times}I$ corresponding to the subgroup $\pi_1(K_1)$ is $p\colon K_1\times I\to \Kl\tilde{\times}I$.\\
    \item Let $T_0=\partial(T^2\tilde{\times}I)$. Then, the $2$-sheeted covering space of $T_0\tilde{\times}I$ corresponding to the subgroup $\pi_1(T_0)$ is $p\colon T_0\times I\to T^2\tilde{\times}I$.
\end{enumerate}

For the case of the \textit{bipod}, there is a non-trivial covering $p\colon Y\to \widehat{B}$. By $(i)$ $\widehat{B}$ is not simply-connected. Construct the covering as in Figure \ref{fig:bipod} by taking $Y=\widehat{D}\cup( K_1\times I)\cup\widehat{D}$ where the unions are over the Klein bottle boundaries and using $(ii)$ in the middle.

For the \textit{tetrapod}, there is a non-trivial covering $p\colon Y\to\widehat{\TP}$. By $(i)$ $\widehat{\TP}$ is not simply-connected. Construct the covering as in Figure \ref{fig:tetrapod}, by taking $Y=\widehat{Q}\cup (T_0\times I)\cup\widehat{Q}$ where the union is over torus boundaries and using $(ii)$ in the middle.
\begin{figure}
    \centering
    \includegraphics[scale=0.6]{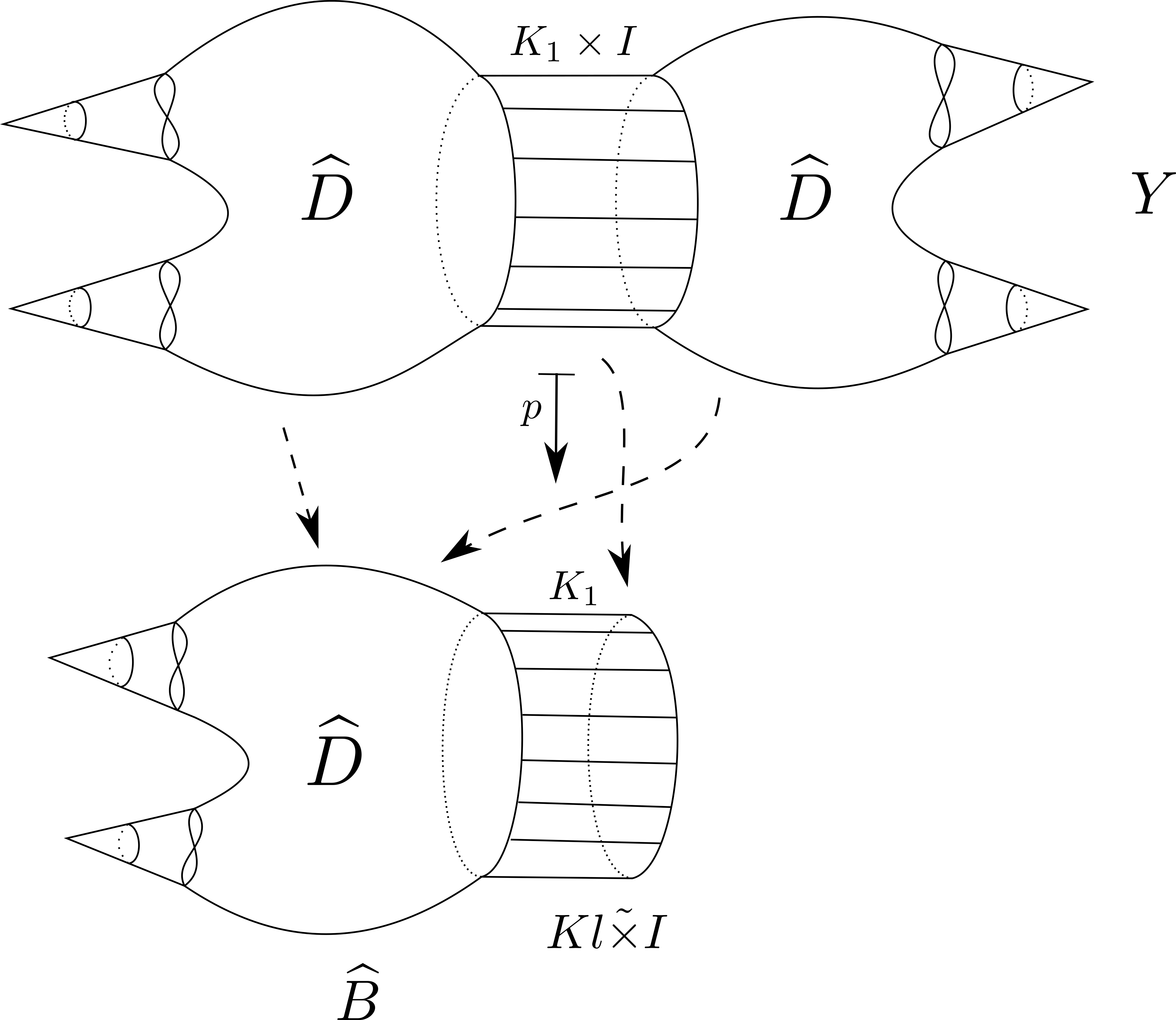}
    \caption{Non-trivial covering for $\widehat{B}$}
    \label{fig:bipod}
\end{figure}
\begin{figure}
    \centering
    \includegraphics[scale=0.4]{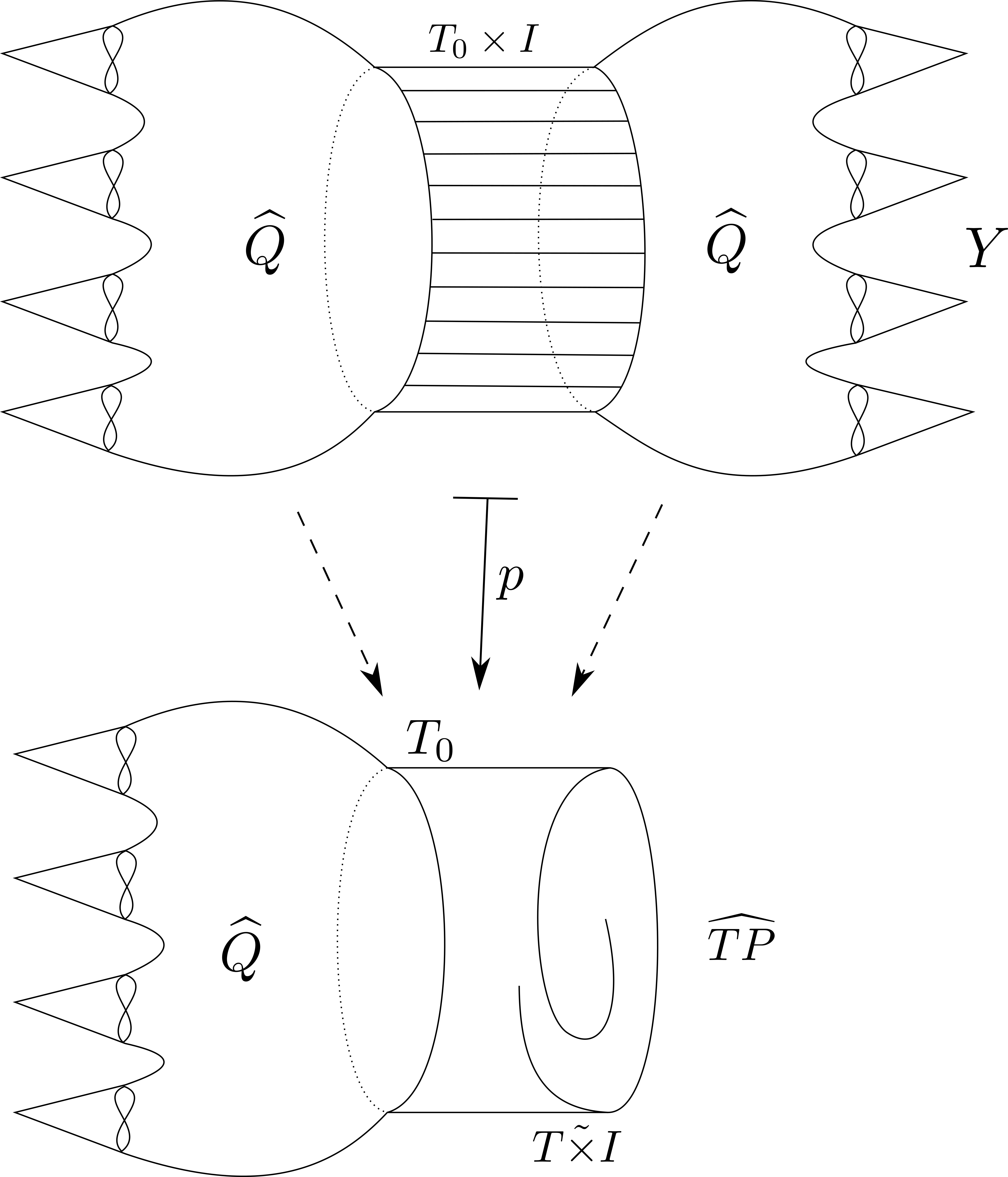}
    \caption{Non-trivial covering for $\widehat{\TP}$}
    \label{fig:tetrapod}
\end{figure}
\end{proof}


\subsection*{Proof of Theorem~\ref{thm:mitsuishi.yamaguchi}}
The spaces listed in Theorem~\ref{thm:mitsuishi.yamaguchi} are $D^3$, $K(P^2)$, $B(\mathrm{pt})$, $B(S_2)$, $B(S_4)$.  The only possible combinations are as follows:
\begin{align*}
    D^3\cup_{S^2}D^3,\ D^3\cup_{S^2}B(S_2),\ K(P^2)\cup_{P^2}K(P^2)\\ B(\mathrm{pt})\cup_{Kl}B(\mathrm{pt}),\ B(S_2)\cup_{S^2}B(S_2),\ B(S_4)\cup_{T^2}B(S_4).
\end{align*}
Then, using Lemmas \ref{double_bpt}, \ref{double_bs2}, and \ref{double_bs4}, we have that these combinations are homeomorphic to
\begin{align*}
    S^3,\ \mathrm{Susp}(P^2),\ \mathrm{Susp}(P^2)\#\mathrm{Susp}(P^2),\text{ or }
    T^3/\beta.
\end{align*}
Let $X$ be a closed simply-connected Alexandrov $3$-space with non-negative curvature. If $X$ is a topological manifold, then, by Perelman's proof of the Poincaré Conjecture, $X=S^3=D^3\cup_{S^2}D^3$. If $X$ has singular points, then, by  \cite[Theorem 1.3]{galaz-guijarro2015}, one of the following assertions holds:
\begin{enumerate}[label=(\alph*)]
    \item $X$ is homeomorphic to $\mathrm{Susp}(P^2)$ which can be realized as an isometric gluing: $D^3\cup_{S^2}B(S_2)$ or $K(P^2)\cup_{P^2}K(P^2)$.\\
    \item $X$ is homeomorphic to $\mathrm{Susp}(P^2)\#\mathrm{Susp}(P^2)$ which can be realized as an isometric gluing: $B(\mathrm{pt})\cup_{Kl}B(\mathrm{pt})$ or $B(S_2)\cup_{S^2}B(S_2)$.\\
    \item $X$ is isometric to a quotient of a closed, orientable, flat $3$-manifold by an orientation-reversing isometric involution with only isolated fixed points.
    By \cite[Theorem 6.7]{luft_sjerve1984},  only three orientable, flat $3$-manifolds admit an orientation-reversing involution with isolated fixed points. These are $T^3$, $T^2\times[0,1]/(z_1,z_2,0)\sim(\overline{z_1},\overline{z_2},1)$, and the Hantzche--Wendt manifold (see Definition \ref{ss:special.manifolds}), which we will denote by $M_6$.
    Moreover, the involutions on these three manifolds are unique up to conjugacy with $8$, $4$, and $2$ fixed points, respectively. The involutions for $T^3$ and $T^2\times[0,1]/(z_1,z_2,0)\sim(\overline{z_1},\overline{z_2},1)$ can be given explicitly. For the $3$-torus $T^3$, we let $\beta$ be the involution given by $\beta(z_1,z_2,z_3)=(\overline{z_1},\overline{z_2},\overline{z_3})$. For $T^2\times[0,1]/(z_1,z_2,0)\sim(\overline{z_1},\overline{z_2},1)$, the involution is given by
    \begin{align*}
        \alpha\colon T^2\times[0,1]/(z_1,z_2,0)\sim(\overline{z_1},\overline{z_2},1)\to T^2\times[0,1]/(z_1,z_2,0)\sim(\overline{z_1},\overline{z_2},1)\\
        \alpha[x,y,z]=\begin{cases}
        [-x,-y,\frac{1}{2}-z]\quad\mathrm{if}\ z\in[0,\frac{1}{2}]\\
        [-\overline{x},-\overline{y},\frac{3}{2}-z]\quad\mathrm{if}\ z\in[\frac{1}{2},1]
        \end{cases}
    \end{align*}
    By \cite[pages 108--109]{kim-sanderson}, $\alpha$ is conjugate to $\tau([z_1,z_2,t])=[-\overline{z_1},\overline{z_2},-t]$. For $M_6$, we will simply denote the corresponding involution by $i_6$. 
    
    The quotient $T^3/\beta$ is the capped octopod, which is simply-connected (see \cite[Proof of Theorem 1.5]{galaz-guijarro2015}); the quotient $T^2\times[0,1]/(z_1,z_2,0)\sim(\overline{z_1},\overline{z_2},1)/\alpha$ is homeomorphic to the tetrapod capped with cones over the projective plane; by Lemma~\ref{bipod_tetrapod}, the capped tetrapod is not simply-connected. Finally, the quotient space $M_6/i_6$ is the bipod capped with cones over the projective plane (see Definition \ref{ss:special.manifolds}); by Lemma \ref{bipod_tetrapod}, the capped bipod is not simply-connected.
\end{enumerate}
\qed


\section{Generalized Dehn surgery}
\label{s:dehn.surgery}

The \emph{Lickorish--Wallace theorem} for $3$-manifolds states that any closed, orientable $3$-manifold may be obtained by performing Dehn surgery on a link in the $3$-sphere \cite{Lickorish1962,Wallace1960}. In the non-orientable case, Lickorish showed that any closed, non-orientable $3$-manifold can be obtained from $S^2\tilde{\times}S^1$, the non-trivial $S^2$-bundle over $S^1$, via surgery on a link \cite[Theorem 3]{Lickorish1963}. Here, we show that any closed non-manifold Alexandrov $3$-space may be obtained  by doing generalized Dehn surgery in $S^2\tilde{\times}S^1$.


\begin{definition}[Generalized Dehn surgery] Let $P$ be a closed Alexandrov $3$-space. A \textit{link} in $P$ is a collection of disjoint knots embedded in $P$. Without loss of generality, we may assume that each of these knots avoids the topologically singular points of $P$; in other words, each knot is in $M_P$, the manifold part of $P$. 
We denote surgery on a link where we allow ourselves to cap off boundary components not only with solid tori or solid Klein bottles but also with copies of $B(\pt)$, as \emph{generalized Dehn surgery}.    
\end{definition}

We are now ready to prove Theorem~\ref{thm:lickorish.dehn.surgery}, which asserts that any closed Alexandrov $3$-space may be obtained by generalized Dehn surgery on a link either in the $3$-sphere or in the non-trivial $S^2$-bundle over $S^1$.

\subsection*{Proof of Theorem~\ref{thm:lickorish.dehn.surgery}}
Since the statement is known in the manifold case, we need only consider topologically-singular 
spaces.
We will show that any closed topologically singular Alexandrov $3$-space may be obtained by generalized Dehn surgery on a link in the non-trivial $S^2$-bundle over $S^1$.

Let $X$ be a topologically singular closed Alexandrov  $3$-space with $2k$ topologically singular points for some $k\geq 1$.
Observe first that we may arrange for each pair $(p_{i}, p_{i+1})$, $k=1,\ldots,k-1$, of topologically singular points to be contained in a  copy of $B(\pt)$ (one for each pair) as follows. The space $X$ is the union of a $3$-manifold $M_X$ with an even number of $ P^2$ boundary components and finitely many cones over these projective planes, corresponding to closed neighborhoods of each topologically singular point $p_i$. In $M_X$, join pairs of the boundary components $ P^2$ by disjoint arcs. A regular neighborhood of a pair of $P^2$ boundary components and its connecting arc is the disk sum of two copies of $P^2\times I$, and attaching the cones to the $P^2$ boundary components gives a $B(\pt)$. After assigning each pair of topologically singular points to a $B(\pt)$, we have $k$ disjoint copies of $B(\pt)$ in $X$, each containing a pair of topologically singular points. Note that the $B(\pt)$ subspaces we get depend on our choice of connecting arc, so we may assign each pair of topologically singular points in $X$ to a $B(\pt)$ in infinitely many ways. 

We now remove the $B(\pt)$ pieces containing pairs of topologically singular points from $X$ to obtain a non-orientable $3$-manifold $N_X$ with $k\geq1$ Klein bottle boundary components. As noted in the preceding paragraph, $B(\pt)$ is the boundary connected sum of two cones over $P^2$, so the boundary of $B(\pt)$ is a Klein bottle. Next, close $N_X$ by gluing in a solid Klein bottle to each of the $k$ Klein bottle boundary components to obtain a closed non-orientable $3$-manifold $N$.

By Lickorish's surgery theorem \cite[Theorem 3]{Lickorish1963}, $N$ can be obtained by surgery on a link in the non-orientable $2$-sphere bundle over $S^1$. By transversality, we can arrange for the link not to intersect the $k\geq 1$ solid Klein bottles added to $N_X$ to obtain $N$.

 Reversing this process, we obtain $X$ by generalized surgery on the non-orientable $S^2$-bundle over $S^1$.
\qed


We obtain Corollary~\ref{cor:lickorish.boundary.4d}, which asserts that every closed Alexandrov $3$-space is homeomorphic to the boundary of a $4$-dimensional Alexandrov $4$-space, as a consequence of the generalized Dehn surgery theorem. This corollary generalizes the classical result that every closed $3$-manifold bounds a $4$-manifold (see \cite{Lickorish1962,Lickorish1963,Thom}).

\subsection*{Proof of Corollary~\ref{cor:lickorish.boundary.4d}}
Let $P$ be a closed Alexandrov $3$-space. Suppose first that $P$ is a manifold. Then there exists a compact $4$-dimensional topological manifold $W$ whose boundary is $P$ (see \cite{Thom} or \cite[Theorem 3]{Lickorish1962}) for the orientable case and  \cite[Theorem 4]{Lickorish1963} for the non-orientable one). This result also holds in the smooth category, ensuring that $W$ is smooth. Since $W$ is smooth and compact, it supports a complete Riemannian metric with sectional curvature uniformly bounded below. Hence, $W$ is a $4$-dimensional Alexandrov space whose boundary is homeomorphic to $P$.  


Suppose now that $P$ is not a manifold. Consider the smooth $4$-dimensional orbifold $Y=D^2\times D^2/\tau$, where $\tau\colon D^2\times D^2\to D^2\times D^2$ is given by 
\[\tau(x,y)=(-x,\overline{y}).\]
This space satisfies, $\partial Y=(S^1\times D^2)\cup(D^2\times S^1)/\tau=B\cup_{Kl}B(\mathrm{pt})$, where $B$ denotes the solid Klein bottle.

Recall from the proof of Theorem~\ref{thm:lickorish.dehn.surgery}
that we may remove a finite number of $B(\mathrm{pt})$ from $P$ to obtain a non-orientable $3$-manifold $N_P$ with a finite number of Klein bottle boundary components. Next, close $N_P$ by gluing a copy of the solid Klein bottle $B$ to each of the Klein bottle boundary components to obtain a closed non-orientable $3$-manifold $V$. By \cite[Theorem 4]{Lickorish1963}, there is a smooth compact $4$-manifold $W$ such that $\partial W=V$. Now, for every solid Klein bottle $B$ that we need to remove from $V$ to construct $P$, glue in a copy of $Y$ to $W$ by identifying the $B$ in $\partial Y$ to the corresponding $B$ in $V$. After smoothing corners, we obtain a compact $4$-dimensional smooth orbifold $Z$ whose boundary is homeomorphic to $P$. Since $Z$ is smooth and compact, it admits a complete orbifold Riemannian metric with sectional curvature uniformly bounded below, which implies that $Z$ is a $4$-dimensional Alexandrov space.
\qed


\bibliographystyle{plainurl}
\bibliography{main}

\end{document}